\documentclass[aap]{imsart}

\RequirePackage{amsthm,amsmath,amsfonts,amssymb, todonotes, verbatim}
\RequirePackage[numbers]{natbib}

\startlocaldefs
\theoremstyle{plain}
\newtheorem{theorem}{Theorem}
\newtheorem{corollary}{Corollary}
\newtheorem{lemma}{Lemma}
\newtheorem{proposition}{Proposition}
\theoremstyle{remark}
\newtheorem{definition}{Definition}
\newtheorem{assumption}{Assumption}
\theoremstyle{remark}
\newtheorem{remark}{Remark} 

\usepackage{enumitem, algorithm, algpseudocode, tikz}
\newcommand{\bl}{}

\usepackage{mathtools}
\DeclarePairedDelimiter\abs{\lvert}{\rvert}%
\DeclarePairedDelimiter\norm{\lVert}{\rVert}%

\makeatletter
\let\oldabs\abs
\def\abs{\@ifstar{\oldabs}{\oldabs*}}
\let\oldnorm\norm
\def\norm{\@ifstar{\oldnorm}{\oldnorm*}}
\makeatother

\usepackage{mathrsfs}
\usepackage{ulem}
\usepackage{hyperref}

\newcommand{\R}{\mathbb{R}}

\newcommand{\N}{\mathbb{N}}

\newcommand{\ud}{\mathrm{d}}

\newcommand{\E}{\mathbb{E}}
\newcommand{\PP}{\mathbb{P}}

\usepackage{bbm} 
\newcommand{\1}{\mathbbm{1}} 
\newcommand\xpm{[x^-,x^+]}

\DeclareMathOperator*{\essinf}{\mathrm{ess inf}}

\endlocaldefs

\begin{document}

\begin{frontmatter}
\title{Optimal stopping with nonlinear expectation: geometric and algorithmic solutions}
\runtitle{Optimal stopping with nonlinear expectation}

\begin{aug}
\author[A]{\fnms{Tomasz}~\snm{Kosmala}\ead[label=e1]{t.kosmala@imperial.ac.uk}\orcid{0000-0001-5450-5002}},
\and
\author[B]{\fnms{John}~\snm{Moriarty}\ead[label=e2]{j.moriarty@qmul.ac.uk}\orcid{0000-0002-4520-4133}}
\address[A]{Department of Mathematics, Imperial College London and School of Mathematical Sciences, Queen Mary University of London\printead[presep={,\ }]{e1}}
\address[B]{School of Mathematical Sciences, Queen Mary University of London\printead[presep={,\ }]{e2}}

\end{aug}

\begin{abstract}
We use the geometry of suitably generalised potentials 
to solve risk-sensitive Markovian optimal stopping problems. As in the linear case due to Dynkin and Yushkievich (1967), the value function is the \bl{pointwise infimum} of those functions which dominate the gain function. An emphasis is placed on \bl{geometric} and pathwise arguments, rather than exploiting convexity, positive homogeneity or related analytical properties. An algorithm is provided to construct the value function at the computational cost of a two-dimensional search.
\end{abstract}

\begin{keyword}[class=MSC]
\kwd[Primary ]{60G40}
\kwd[; secondary ]{91B08}
\kwd{91B06}
\kwd{60J25}
\end{keyword}

\begin{keyword}
\kwd{Optimal stopping}
\kwd{nonlinear expectation}
\kwd{risk measure}
\kwd{constructive solution}
\kwd{Markov property}
\end{keyword}

\end{frontmatter}

Convexity and associated properties have been used as main tools in the global solution of both stochastic and deterministic optimisation problems. This paper concerns maximising a Markovian performance criterion when stopping a diffusion,
where the relationship between concavity and excessivity has been used since the seminal work of Dynkin \cite{dynkin1963optimum} (see also \cite{Dayanik_Karatzas}). However this concavity arises from the linearity of the expectation operator, and there is increasing recent interest in the use of nonlinear, including nonconcave, expectations. In this paper we \bl{study the lower envelope (that is, the pointwise infimum) of suitably generalised potentials dominating the gain function (Figure \ref{fig:1}),}    
to obtain constructive formulae for the value function and optimal stopping policy in the Markovian context. The formulae are thus analogues to those of Dynkin and Yushkievich \cite{Dynkin_Yushkevich} for the linear case.

Interest in such nonlinear criteria has arisen in numerous domains. A few examples include the theory of preferences, see for example \cite[Ch. 2]{Follmer_Schied} and \cite{Wang_Xu}, optimisation under ambiguity \cite{chen2002ambiguity, Cheng_Riedel, Christensen_2013}, and even reinforcement learning \cite{Huang_Haskell, lam2023riskaware}. The celebrated {\it prospect theory} of Kahneman and Tversky \cite{kahneman1979prospect} describes how the same agent can switch between risk averse (concave) and risk seeking (convex) criteria, depending on whether losses or gains are considered. This helps to motivate methods based on other properties.

In the absence of convexity or related global analytical properties it is natural to seek local solutions. Accordingly we study the \bl{lower envelope} 
of a constrained set of generalised potentials to solve the problem locally. A pathwise argument then confirms this local solution to be optimal. 

Our main result is that, under sufficient regularity, the value of the nonlinear optimal stopping problem is the pointwise infimum over those functions in $\mathcal{H}$ which dominate the gain function determining the reward or payoff. It is optimal to stop the first time this value function coincides with the gain function. Thus the value function and optimal stopping region are constructed from the nonconvex geometry of the generalised potentials and gain function. 
In the linear case, this result was obtained for the Wiener process in \cite[Ch.\ III.7]{Dynkin_Yushkevich}, and extended to one-dimensional regular diffusion processes (with exponential discounting of the reward) in \cite{Dayanik_Karatzas} by appropriately generalising the notion of concavity. For $\kappa$-ambiguity, which is a positively homogeneous expectation, a result of this type was obtained in \cite{Christensen_2013} using ideas from \cite{beibel1997new}. 

A common solution technique in optimal stopping problems is to generate a  heuristic solution and then apply a verification argument (see for example \cite[Th. 3.4]{Cheng_Riedel}). Although the assumptions on the nonlinear expectation are strong (see for example \cite{kupper2009representation}), in Section \ref{sec:motiv} we confirm that the solution is correct for the worst-case risk mapping, which is less regular. If used as a heuristic then, as illustrated by Figure \ref{fig:4}, the construction differs from the well-known principle of smooth fit or smooth pasting (see for example \cite{Peskir2006}). As a corollary, no smooth transformation of Figure \ref{fig:4} could reduce the solution to the classical one of \cite[Ch.\ III.7]{Dynkin_Yushkevich}.

To the best of our knowledge the method of proof is different to existing approaches using linearity, generalised concavity, positive homogeneity or related analytical properties, and is novel even in the linear case. This may be of independent interest, particularly as few explicit solutions are known for optimal stopping of Brownian motion in dimension greater than one \cite{christensen2019optimal}, although the topic is outside the scope of this paper. 

A simple algorithmic implementation of the construction returns the value function and optimal stopping region, at the cost of a two-dimensional nonconvex optimisation. The algorithm is presented under regularity assumptions which may be compared to those made in \cite{Crocce_Mordecki}, where the authors use analytic methods to provide an algorithmic solution in the linear case.

\section{Optimal stopping with nonlinear expectation}

\begin{figure}
    \centering
        \resizebox{10cm}{!}{
\begin{tikzpicture}
\node[anchor=south west,inner sep=0] (image) at (0,0) {\includegraphics{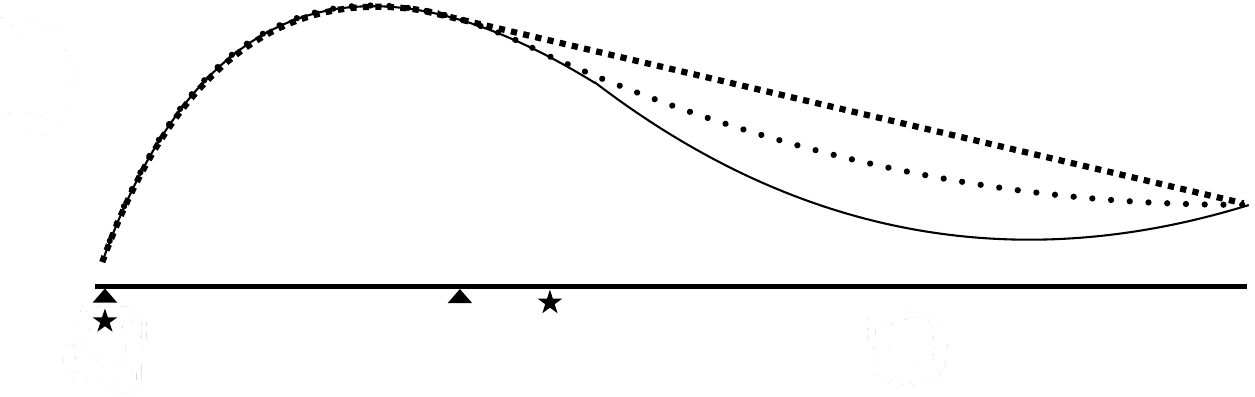}};
\begin{scope}[x={(image.south east)},y={(image.north west)}]
 \node[scale=2] at (0.082,0.072) {$0$};
 \node[scale=2] at (0.992,0.072) {$1$};
 \node[scale=2] at (0.92,0.38) {$g$};
 \node[scale=2] at (0.73,0.5) {$V^\text{nlin}$};
 \node[scale=2] at (0.6,0.86) {$V^\text{lin}$};
\end{scope}
\end{tikzpicture}
}
    \caption{Solutions for an example gain function $g$ (solid curve) when $X^x$ is Brownian motion. The value function $V^\text{lin}$ in the linear case $\rho=\E$ is the smallest concave majorant of $g$ (square markers). The value function $V^\text{nlin}$ for an indicative nonlinear $\rho$ is a nonconcave majorant of $g$ (circular markers). The optimal stopping regions are respectively the intervals marked by triangles and stars. In the differentiable case, the value and gain functions meet smoothly (at the rightmost triangles and stars respectively, see Corollary \ref{cor:sf}); for a non-smooth case see Figure \ref{fig:4}.}
    \label{fig:1}
\end{figure}

\label{sec:statement}
The nonlinear optimal stopping problem we consider is
\begin{equation}\label{optimal_stoping_intro}
    V(x):=\sup_{\tau\in\mathscr{T}} \rho(g(X^x_\tau)),
\end{equation}
where $\rho$ is a nonlinear expectation (more precisely a {\it risk mapping} with a reference probability measure $\PP$, as defined in Section \ref{subsec:risk_mappings} below), $g \colon [0,1] \to \R_+$ is a {\it gain function}, $X^x$ is a sufficiently regular diffusion process which starts at $x \in [0,1]$ and is absorbed on $\{0, 1\}$, $\mathscr{T}$ is the set of $\PP$-almost surely finite stopping times, and $V$ is the {\it value function}. 
Our approach is via the elementary geometry of functions such as $x \mapsto h^{y,z}_{\beta,\gamma}(x)$, where 
\begin{align*}
h^{y,z}_{\beta,\gamma}(x) := 
\begin{cases}
    \rho(\beta \1_{\{\tau_y^x < \tau_z^x\}} + \gamma \1_{\{\tau_y^x > \tau_z^x\}}), &  0 \leq y < x < z \leq 1, \\
    \infty, & \text{ otherwise, }
\end{cases}
\end{align*}
and $\tau_a^x$ denotes the first hitting time by $X^x$ of the point $a \in [0,1]$. (Where no confusion arises, we may omit the superscript $x$ and write $\tau_a$.) By analogy with the case of linear expectation, where they are solutions to the Dirichlet problem,  we refer to the functions $h^{y,z}_{\beta,\gamma}$ as {\it generalised potentials}. In Section \ref{sec:2} we obtain the solution
\begin{align}\label{eq:guessv}
V(x) = \inf \{ h(x): h \in \mathcal{H}, \; h \geq g \text{ on } [0,1]\},
\end{align}
where the set $\mathcal{H}$ is given by
\begin{multline}\label{eq:defcalh}
    \mathcal{H} =  \Big\{  h_{\beta,\gamma}^{y,z} : \left((y=0, z=1
    ) \text{ or }
        (y \in (0,1), \beta = \bar g 
        \text{ and } z = 1) \right. \\ \left. \text{ or }
        (y = 0, z \in (0,1) \text{ and } \gamma = \bar g 
        )\right)
        \Big\}, 
    \end{multline}
where $\bar g = \max_{x \in [0,1]} g(x)$. In the classical case \cite[Ch.\ III.7]{Dynkin_Yushkevich} (Figure \ref{fig:1}, square markers), the expectation is $\rho = \E$ and the functions $h^{y,z}_{\beta,\gamma}$ are linear. In the present setting (Figure \ref{fig:1}, circular markers), the functions $h^{y,z}_{\beta,\gamma}$ are nonlinear and defined to be finite only on a subinterval of $[0,1]$. This choice for $\mathcal{H}$ is motivated by the example in Section \ref{sec:motiv}.

\begin{remark} \label{rem:constr}
Definition \eqref{eq:defcalh} is used for simplicity. In fact it is sufficient to use a set of functions smaller than $\mathcal{H}$ in the construction, by adding to \eqref{eq:defcalh} the constraint that $\beta \vee \gamma := \max\{\beta, \gamma\} \leq \bar g$, since the infimum in \eqref{eq:guessv} then remains unchanged. This will be an easy consequence of the monotonicity of $\rho$ (Definition \ref{def:dynamic_risk_mapping}), the $\varrho$-martingale property (Lemma \ref{lem:mgale}), and the continuity of $\mathcal{H}$ (Lemma \ref{lem:hprops})). 
\end{remark}

\subsection{A motivating example}
\label{sec:motiv}

A key idea in this work is to define the generalised potentials not only on $[0,1]$, but also on some subdomains of $[0,1]$ (that is, to allow either $y>0$ or $z<1$) in \eqref{eq:defcalh}.) Since this contrasts with the linear case (see eg. \eqref{eq:hdag} below), the following example provides useful motivation.
For bounded random variables $Z$ consider the worst-case risk mapping $\rho$ given by
\begin{align*}
    \rho(Z) = \PP - \essinf (Z).
\end{align*}
With the above setup, the generalised potentials are
\begin{align}\label{eq:wcrm}
h^{y,z}_{\beta,\gamma}(x) = 
\begin{cases}
    \beta, & x = y, \\
    \gamma, & x = z, \\
    \beta \wedge \gamma, & x \in (y,z), \\
    \infty, & \text{ otherwise. }
\end{cases}
\end{align}
Considering the constraints which define $\mathcal{H}$ (cf. \eqref{eq:defcalh}), for example by a suitable sketch, it is easy to see that 
\begin{align}\label{eq:wcvf}
\inf \{ h(x): h \in \mathcal{H}, \; h \geq g \text{ on } [0,1]\} = \min\left\{\sup_{y \in [0,x]}g(y), \sup_{y \in [x,1]}g(y)\right\},
\end{align}
see for example Figure \ref{fig:4}.
\begin{figure}
    \centering
        \resizebox{10cm}{!}{
\begin{tikzpicture}
\node[anchor=south west,inner sep=0] (image) at (0,0) {\includegraphics{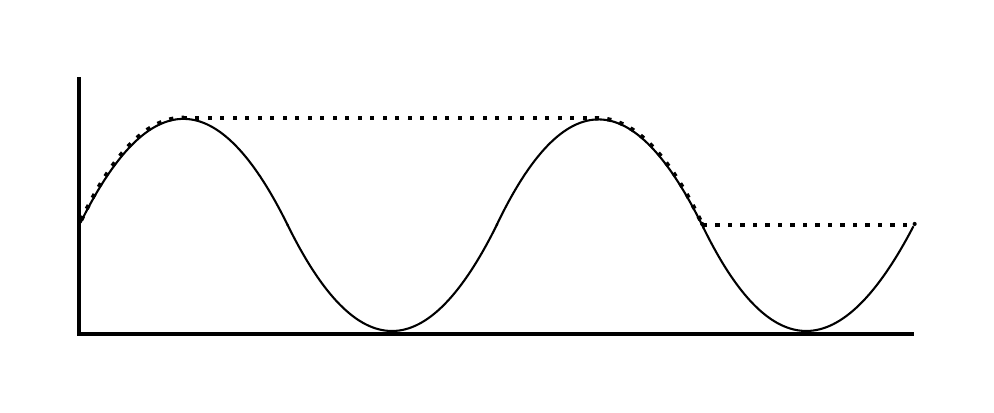}};
\begin{scope}[x={(image.south east)},y={(image.north west)}]
 \node[scale=2] at (0.082,0.07) {$0$};
 \node[scale=2] at (0.92,0.07) {$1$};
 \node[scale=2] at (0.18675,0.07) {$\frac 1 8$};
 \node[scale=2] at (0.60575,0.07) {$\frac 5 8$};
 \node[scale=2] at (0.7105,0.07) {$\frac 3 4$};
 \node[scale=2] at (0.255,0.5) {$g$};
 \node[scale=2] at (0.4,0.78) {$V$};
\end{scope}
\end{tikzpicture}
}
    \caption{Solution for the gain function $g(x) = 1 + \sin (4 \pi x)$ under the worst-case risk mapping. Although this nonlinear expectation has less regularity than assumed below, the value function (square markers) is that of Theorem \ref{pro:representation_4param}, namely $V(x)=\inf \{ h(x): h \in \mathcal{H}, \; h \geq g \text{ on } [0,1]\}$. The value and gain functions meet smoothly at $x = \frac 1 8$ and $x = \frac 5 8$, but not at $x=\frac 3 4$. No smooth transformation of this figure could reduce the solution to the classical one illustrated in Figure \ref{fig:1}.}
    \label{fig:4}
\end{figure}
We may now verify that \eqref{eq:wcvf} is the value function. Taking $x_1 \in \arg \sup_{y \in [0,x]}g(y)$ and $x_2 \in \arg \sup_{y \in [x,1]}g(y)$, 
the stopping time $\tau_{x_1,x_2} = \min\{\tau_{x_1}, \tau_{x_2}\}$ gives
\[
V(x) \geq \rho\left(g\left(X^x_{\tau_{x_1,x_2}}\right)\right) = \min\left\{\sup_{y \in [0,x]}g(y), \sup_{y \in [x,1]}g(y)\right\}.
\]
For the reverse inequality take any stopping time $\tau \in \mathscr{T}$ and, writing $S$ for the scale function of $X^x$ (see e.g.\ \cite[V.28]{Rogers_Williams_2}), note that $X^x_\tau \in [0,x]$ with positive probability since $S(X^x)$ is a martingale. It follows that $\PP - \essinf (g(X^x_\tau)) \leq \sup_{y \in [0,x]}g(y)$. Since $\tau$ was arbitrary, the symmetric argument on $[x,1]$ gives
\[
V(x) \leq \min\left\{\sup_{y \in [0,x]}g(y), \sup_{y \in [x,1]}g(y)\right\},
\]
and so the solution \eqref{eq:guessv} holds.

\subsection{Probabilistic setting}\label{sec:setting}
As noted in \cite[Ch.\ III.7]{Dynkin_Yushkevich}, problem \eqref{optimal_stoping_intro} would be trivial if $X^x$ were defined
on the real line, since it then has probability one of hitting any point. This motivates the choice of $X^x$ as a diffusion in the interval $[0,1]$ with absorption at its endpoints.

Take the classical Wiener space $\Omega = C(\R_+;\R)$ with $B_t(\omega)=\omega(t)$ for $\omega \in \Omega$ and $t \geq 0$, taking $\mathbb{F}=(\mathcal{F}_t)_{t\geq 0}$ to be the smallest filtration satisfying the usual conditions constructed from the natural filtration $\mathcal\{\sigma(B(s): s \in [0,t])\}_{t \geq 0}$, and $\mathcal{F} = \sigma \left( \cup_{t \geq 0} \mathcal{F}_t \right)$. Let $\PP$ be a probability measure on $(\Omega,\mathcal{F})$ under which $B = (B_t)_{t \geq 0}$ is a standard Brownian motion on $\R$ with $B_0=0$ almost surely (we will refer to $\PP$ as the {\it reference probability}). Let $\mu$, $\sigma$ be real measurable functions with $\sigma$ taking strictly positive values such that, for each $x\in \R$, there exists a solution $Y^x := (Y^x_s(\omega))_{s \geq 0}$ to the stochastic differential equation 
\begin{align}\label{eq:X^x}
    dY^x_s = \mu(Y_s^x)ds + \sigma(Y_s^x)dB_s, \qquad s \geq 0,
    \qquad Y_0^x = x \quad \PP\text{--a.s.,}
\end{align}
and such that, for each $t>0$ and almost every $\omega \in \Omega$, the path $Y^x_{0,t} := (Y^x_s(\omega))_{0 \leq s \leq t}$ depends continuously on $x$ in the uniform topology. Set
\begin{align}\label{eq:stopped_Y}
    X^x_t := Y^x_{t \wedge \tau_{0,1}}, \qquad t \geq 0,
\end{align}
where, for $a \in [0,1]$, we define $\tau_a^x = \inf \{ t\geq 0: Y^x_t=a \}$, and for $0 \leq a < b \leq 1$ set $\tau_{a,b}^x = \inf \{ t\geq 0: Y_t^x \notin (a,b)\}$, and write $\tau_a$ and $\tau_{a,b}$ respectively if no confusion arises.
That is, the process $X^x = (X^x_t)_{t \geq 0}$ starts at $x$ and is absorbed at the boundary $\{0,1\}$. For $t\geq 0$ let $\theta_t$ be the shift operator $\theta_t \colon \Omega \to \Omega$ given by 
\begin{align}\label{eq:theta}
        \theta_t(\omega)(s) &= \omega(t+s), \qquad s \geq 0.  
\end{align}

\begin{remark}\label{rem:cd}
    This continuous dependence of the path $Y^x_{0,t}$ on $x$ in the uniform topology is trivial for the Brownian motion $Y^x := B + x$. 
    For more general diffusions, a simple sufficient condition is that the coefficients are Lipschitz, see e.g. \cite[Lem.\ 21.5 and p.\ 417]{Kallenberg}.
\end{remark}

\subsection{Risk mappings}
\label{subsec:risk_mappings}

Although the space of all nonlinear performance criteria is large, a number of basic properties are desirable when studying optimisation problems, see for example \cite{Cheng_Riedel}. We adopt the following commonly used general dynamic framework, which (modulo possible sign changes) has been called a conditional risk measure \cite{Follmer_Schied}, dynamic risk measure \cite{kupper2009representation} and dynamic conditional risk mapping \cite{Kosmala_Martyr_Moriarty}. Let $b\mathcal{F}$, $b\mathcal{F}_t$ denote respectively the spaces of bounded random variables and bounded $\mathcal{F}_t$-measurable random variables on $(\Omega,\mathcal{F})$.

\begin{definition}\label{def:dynamic_risk_mapping}
A {\it dynamic risk mapping} is a family
$\varrho = \{ \rho_t : t \geq 0\}$
with $\rho := \rho_0$
of functions $\rho_t \colon b\mathcal{F} \to b\mathcal{F}_t$
for $t>0$ and $\rho \colon b\mathcal{F} \to \R$
such that 
\begin{enumerate}
\item $\rho$
is a risk mapping, that is it satisfies:
\begin{itemize}
\item normalisation: $\rho(0)=0$,
\item monotonicity: for all $Y,Z \in b\mathcal{F}$ we have $Y \leq Z$ $\PP$-a.s.\ $\implies \rho(Y)\leq \rho(Z)$,
\item translation invariance: $Y\in b\mathcal{F}, c \in \R \implies \rho(Y+c) = \rho(Y)+c$,
\end{itemize}
\item for each $t>0$ $\rho_t$ is a conditional risk mapping, that is it satisfies:
\begin{itemize}
\item normalisation: $\rho_t(0)=0$, $\PP$-a.s.
\item monotonicity: for all $Y,Z \in b\mathcal{F}$ we have $Y \leq Z$ $\PP$-a.s.\ $\implies \rho_t(Y)\leq \rho_t(Z)$ $\PP$-a.s.,
\item conditional translation invariance: $Y\in b\mathcal{F}, Z\in b\mathcal{F}_t \implies \rho_t(Y+Z) = \rho_t(Y)+Z$ $\PP$-a.s..
\end{itemize}
\end{enumerate}
\end{definition}

\noindent A consequence is {\it conditional locality} (see for example \cite[Prop.\ 3.3]{Cheridito_Delbaen_Kupper_dyn}):
 for every $Z$, $Z'$ in $b\mathcal{F}$ and $A \in \mathcal{F}_t$,
\[
    \rho_t(\1_{A}Z + \1_{A^{c}}Z') = \1_{A}\rho_t(Z) + \1_{A^{c}}\rho_t(Z').
\]

\begin{definition}\label{def:hbasics}
The dynamic risk mapping $\varrho$ is
\begin{itemize}
\item  {\it continuous} if for each bounded sequence $\{Y_n\}_{n \geq 1} \subset b\mathcal{F}$ with $Y_n \to Y$ $\PP$-a.s.\ we have 
\[
\rho_t(Y_n) \to \rho_t(Y) \qquad \PP\text{-a.s.\ for all }t\geq 0,
\]
\item {\it time consistent} if for all bounded stopping times $\tau_1\leq \tau_2$ we have 
\begin{equation}\label{eq:time_consistency}
\rho_{\tau_1} = \rho_{\tau_1} \circ \rho_{\tau_2}, 
\end{equation}
(cf. \cite[Def. 2]{bion2009time}),
\item {\it concave} if for all $t \geq 0$, $\lambda \in [0,1]$ and $Y,\; Z \in b\mathcal{F}$ we have 
\[
\rho_t(\lambda Y + (1-\lambda)Z) \geq \lambda \rho_t(Y) + (1-\lambda)\rho_t(Z), 
\]
(see for example \cite{Delbaen_Peng_Rosazza_Gianin, Pflug_Romisch}),
\item {\it positively homogeneous} if $\rho_t(\lambda Z) = \lambda \rho_t(Z)$ for all 
$t \geq 0$ and $\lambda \geq 0$,
\item {\it law invariant} if for each 
$t \geq 0$ we have $\rho_t(Y) = \rho_t(Z)$ if $Y,Z \in b\mathcal{F}$ have the same distribution under $\PP$.
\end{itemize}
\end{definition}
\noindent Recalling that for each $\omega \in \Omega$ we have $X^x(\omega):=(X^x_t)_{t \geq 0}(\omega) \in C(\R_+;\R)$, we will appeal to the probabilistic Markov property introduced in \cite{Kosmala_Martyr_Moriarty}.

\begin{definition}\label{def:mp}
    The dynamic risk mapping $\varrho$ is {\it Markov} if for all $x \in [0,1]$, $t>0$ and $Z = Z(X^x) \in b\mathcal{F}$ we have  
        \[
        \rho_t\left(Z\left(X^{X_t^x}\right) \circ \theta_t \right) = \rho\left(Z\left(X^{X_t^x}\right)\right) \qquad \PP\text{-a.s.,}
        \]    
and {\it strong Markov} if for each $\tau \in \mathscr{T}$, $x \in [0,1]$ and $Z \in b\mathcal{F}$ we have  
        \[
        \rho_\tau\left(Z\left(X^{X_\tau^x}\right) \circ \theta_\tau \right) = \rho\left(Z\left(X^{X_\tau^x}\right)\right) \qquad \PP\text{-a.s..}
        \]
\end{definition}

\subsection{Related work}
\label{subsec:literature}

Several studies have been made of {\it robust} optimal stopping, where multiple possible linear expectations are considered. The stopping policy is chosen assuming the worst case, making these robust criteria concave. Related work in the Markovian setting includes \cite{Alvarez}, and \cite{Cheng_Riedel} in higher dimension, where solutions are characterised using the Bellman equation. 
In the non-Markovian framework (also in higher dimension), martingale methods are applied to characterise solutions in terms of Snell envelopes, see for example \cite{Bayraktar_Yao_part_I,Bayraktar_Yao_part_II,Jelito_Pitera_Stettner,Nutz_Zhang}. Such problems are  also closely related to stochastic games of control and stopping, see for example \cite{Karatzas_Zamfirescu}.

To the best of our knowledge, algorithmic solutions have thus far only been provided in the linear case. In addition to \cite{Crocce_Mordecki} there is a series of studies \cite{cho2002linear, helmes2010construction, rohl2001linearer} which embed the problem within an infinite dimensional linear program over a space of measures and examine its dual, yielding a semi-infinite linear program and a nonlinear optimisation problem which facilitate the construction of the value function. Our algorithm has some similarity to the latter approach, in that it begins by seeking local solutions.

\vspace{2mm}
The remainder of the paper is organised as follows. Section \ref{subsec:assum} states assumptions and provides preliminary results, Section \ref{sec:repvf} states and proves the main result, Section \ref{subsec:stopset} establishes properties of the stopping set and the regularity of the value function, Section \ref{subsec:algorithm} develops the solution algorithm, and Section \ref{sec:gt} provides examples.

\section{Geometric and algorithmic solutions}\label{sec:2}

Motivated by the remarkable constructive result of \cite{Dynkin_Yushkevich}, our aim in this work is to identify a family $\mathcal{H}$ of generalised potentials whose geometry encodes the solution to \eqref{optimal_stoping_intro} in similar fashion. In the former case, one takes the set $\mathcal{H}^\dagger$ of all affine functions on $[0,1]$:
\begin{align}\label{eq:hdag}
        \mathcal{H}^\dagger := \left\{ h^{0,1}_{\beta,\gamma} : \beta,\gamma \in \R\right\},
    \end{align}
and the value function is the pointwise infimum over those $h \in \mathcal{H}^\dagger$ which dominate the gain function. Although the structure of the set $\mathcal{H}^\dagger$ is attractively simple, the counterexamples of Section \ref{sec:gt} will demonstrate that its use does not generate the correct solution in general.

In the classical proof of \cite[Ch.\ III.7]{Dynkin_Yushkevich}, the problem is solved globally using the equivalence of excessive and nonnegative concave functions, or the `fundamental property of concave functions' of \cite[Ch.\ III.8]{Dynkin_Yushkevich}. Our method of proof is instead local. Given $x = X_0$, we identify the first time at which it is optimal to stop. Under this local approach, as noted in \cite{cho2002linear, helmes2010construction}, it is more natural to consider functions $h^{y,z}_{\beta,\gamma}$ with $0 \leq y \leq x \leq  z \leq 1$. The set $\mathcal{H}$ will include some, although not all such functions, as follows. 

The argument below requires that at each $v \in \{y, z\}$, either the process is absorbed (i.e. $v \in \{0,1\}$) or $h(v) \geq \bar g = \max_{x \in [0,1]} g(x)$ (see the proof of Lemma \ref{lem:parti}). This suggests the following construction for $\mathcal{H}$ (which is illustrated in Figure \ref{fig:3}):

\begin{figure}
    \centering
\resizebox{11cm}{!}{
\begin{tikzpicture}
\node[anchor=south west,inner sep=0] (image) at (0,0) {\includegraphics{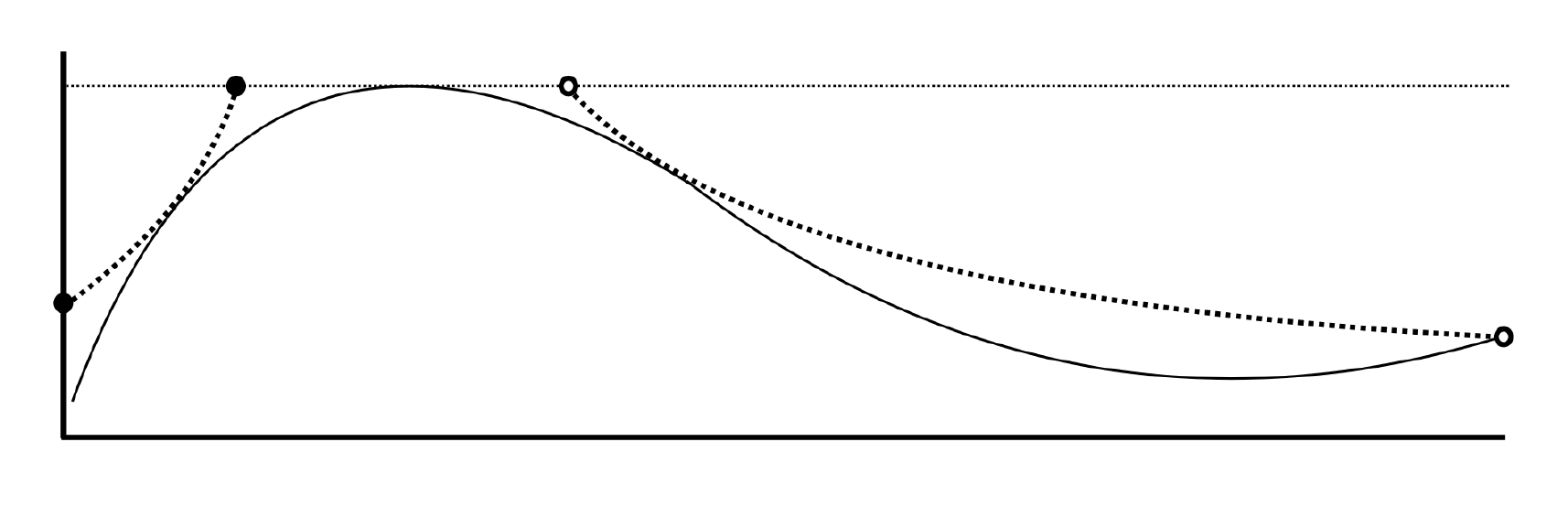}};
\begin{scope}[x={(image.south east)},y={(image.north west)}]
 \node[scale=3] at (0.04,0.072) {$0$};
 \node[scale=3] at (0.12,0.072) {$x_1$};
 \node[scale=3] at (0.42,0.072) {$x_2$};
 \node[scale=3] at (0.96,0.072) {$1$};
 \node[scale=3] at (0.02,0.83) {$\bar g$};
 \node[scale=2.5] at (0.09,0.76) {$h^{y_1,z_1}_{\beta_1,\gamma_1}$};
 \node[scale=2.5] at (0.75,0.51) {$h^{y_2,z_2}_{\beta_2,\gamma_2}$};
\end{scope}
\end{tikzpicture}
}
    \caption{Two generalised potentials $h^{y_i,z_i}_{\beta_i,\gamma_i}$, $i=1,2$, satisfying $h^{y_i,z_i}_{\beta_i,\gamma_i} \geq g$. Both lie in $\mathcal{H}$ (cf. Definition \ref{def:hcal}), since their endpoints satisfy (i) $y_1=0$, $\gamma_1 = \bar g$ (filled circular markers) and (ii) $z_1=1$, $\beta_1 = \bar g$ (open markers). When evaluated at the point $x=x_1$, $h^{y_1,z_1}_{\beta_1,\gamma_1}$ achieves the infimum in \eqref{eq:guessv}, and $h^{y_2,z_2}_{\beta_2,\gamma_2}$ achieves the corresponding infimum at $x=x_2$.}
    \label{fig:3}
\end{figure}

\begin{definition}\label{def:hcal}
\begin{enumerate}
    \item (The set $\mathcal{H}$) \label{def:hcal1}
    Write 
    \begin{multline*}
    \mathcal{H} =  \Big\{  h_{\beta,\gamma}^{y,z} : \left((y=0, z=1) \text{ or }
        (y \in (0,1), \beta = \bar g 
        \text{ and } z = 1) \right. \\ \left. \text{ or }
        (y = 0, z \in (0,1) \text{ and } \gamma = \bar g 
        )\right)
        \Big\} 
    \end{multline*}
    where 
    \begin{align*}
h^{y,z}_{\beta,\gamma}(x) := 
\begin{cases}
    \rho(\beta \1_{\{\tau_y^x < \tau_z^x\}} + \gamma \1_{\{\tau_y^x > \tau_z^x\}}), &  0 \leq y < x < z \leq 1, \\
    \infty, & \text{ otherwise, }
\end{cases}
\end{align*}
    \item (Domains) Write $d\left(h_{\beta,\gamma}^{y,z} \right) := (y,z)$ and $\overline{d\left(h_{\beta,\gamma}^{y,z} \right)} := [y,z]$,  
\item ($\Delta$-Equicontinuity) \label{equicon} 
The set $\mathcal{H}$ will be called
{\it $\Delta$-equicontinuous} if for each $\Delta >0$, the functions 
\[
\{h_{\beta,\gamma}^{y,z}: 0 \leq y < z \leq 1, \; z-y > \Delta,
\; \beta, \gamma \in [0,\bar g + 1]\}
\]
are equicontinuous. 

    That is: 
    for each $\Delta > 0$ and each such function $h = h_{\beta,\gamma}^{y,z}$,
    there exists $\delta > 0$ such that
    \[
        |u - v| \leq \delta \implies |h(u)-h(v)| \leq \epsilon,
    \] 
    \item ($\varrho$-martingales) A process $M = (M_t)_{t\geq 0}$ is a {\it $\varrho$-martingale} 
    if $M$ is bounded and for each $t \geq s$ 
    we have $\rho_s(M_t)=M_s$, $\PP$-a.s.. 
\end{enumerate}
\end{definition}

\begin{remark}\label{rem:equicontinuity}
    On each compact domain the properties of pointwise equicontinuity and uniform equicontinuity coincide \cite[Th.\ 36, p.\ 261]{Pugh}, and we do not distinguish between the two. Also, we will say `$\mathcal{H}$ is differentiable' if each $h\in \mathcal{H}$ is differentiable.
\end{remark}

For $y < z \in [0,1]$ and $\beta, \gamma \in [0, \bar g]$ define 
\begin{align*}
    g_{\beta,\gamma}^{y,z}(x) &= \beta \1_{\{y\}}(x) + \gamma \1_{\{z\}}(x), \qquad x \in [0,1].
\end{align*}

\subsection{Assumptions}

The standing assumption made in this paper is the following:

\begin{assumption}\label{continuity_assumption}
    The dynamic risk mapping $\varrho$ is strongly Markov, time consistent and continuous, $\mathcal{H}$ is $\Delta$-equicontinuous, and $g$ is continuous and nonnegative.
\end{assumption}

The $\Delta$-equicontinuity of $\mathcal{H}$ is assumed in order to obtain the continuity of the \bl{lower envelope} 
\begin{align}\label{eq:repV}
w(x) &:= \inf \{ h(x): h \in \mathcal{H}, \; h \geq g \text{ on } [0,1]\},
\end{align}
(see Figure \ref{fig:1}). In line with the classical case (cf. \eqref{eq:hdag}), each $h \in \mathcal{H}$ will be continuous (cf. part \ref{ph0} of Lemma \ref{lem:hprops} below), so the \bl{lower envelope} 
$w$ (as the infimum of a family of upper semicontinuous functions) is upper semicontinuous.

\vspace{1mm}
Note that by translation invariance, the assumption that $g$ is nonnegative is made without loss of generality. It is also interesting to note that the worst-case risk mapping of Section \ref{sec:motiv} is outside the scope of Assumption \ref{continuity_assumption}, since $\mathcal{H}$ contains discontinuous functions in that case.

\subsection{Preliminary results}\label{sec:pr}
Henceforth making Assumption \ref{continuity_assumption}, we obtain the following results:

\begin{lemma}\label{lem:mgale} Writing $M_t^x = h_{\beta,\gamma}^{y,z}(X_{t\wedge\tau_{y,z}}^x)$, the process $(M_t^x)_{t \geq 0}$ is a $\varrho$-martingale, where $0 \leq y \leq x \leq z \leq 1$. In particular $h_{\beta,\gamma}^{y,z}(x) = \rho \left(h_{\beta,\gamma}^{y,z}(X_{\tau_{y,z}}^x)\right)$.
\end{lemma}

\begin{proof} 
Setting $g=g_{\beta,\gamma}^{y,z}$ and $h=h_{\beta,\gamma}^{y,z}$, we have by Lemma \ref{lem:mgale} and the strong Markov property for $\varrho$ that, for $s \leq t$,
\begin{align*}
\rho_s\left(M_t^x\right)
&= \rho_s \left(h\left(X_{t\wedge \tau_{y,z}}^x\right) \right) 
= \rho_s \left(\rho\left(g\left(X_{\tau_{y,z}}^{X_{t\wedge \tau_{y,z}}^x}\right)\right) \right) \\
&= \rho_s \left( \rho_{t\wedge\tau_{y,z}}\left(g\left(X_{\tau_{y,z}}^{X_{t\wedge \tau_{y,z}}^x} \circ \theta_{t\wedge\tau_{y,z}}\right)\right) \right) \\
&= \rho_{s\wedge\tau_{y,z}}\left(g\left(X_{\tau_{y,z}}^{X_{t\wedge \tau_{y,z}}^x} \circ \theta_{t\wedge\tau_{y,z}}\right)\right) = \rho_{s\wedge\tau_{y,z}}\left(g\left(X_{\tau_{y,z}}^{X_{s\wedge \tau_{y,z}}^x} \circ \theta_{s\wedge\tau_{y,z}}\right)\right)\qquad (*)\\
&= \rho \left(g\left(X_{\tau_{y,z}}^{X_{s\wedge\tau_{y,z}}^x}\right)\right) = h\left(X_{s\wedge\tau_{y,z}}^x    \right) = M_s^x,
\end{align*}
\noindent where in step $(*)$ we use the following properties: firstly, for all $s \geq 0$ we have 
\begin{align*}
    X_{\tau_{y,z}}^x &= X_{\tau_{y,z}}^{X_{s\wedge \tau_{y,z}}^x} \circ \theta_{s\wedge\tau_{y,z}},
\end{align*}
and secondly, if $\tau_1$ and $\tau_2$ are two stopping times, then
$$\rho_{\tau_1} \circ \rho_{\tau_2} = \rho_{\tau_1\wedge\tau_2}.$$
Indeed by conditional translation invariance, conditional locality and time consistency, for each $Z \in b\mathcal{F}$ we have
\begin{align*}
\rho_{\tau_1}(\rho_{\tau_2}(Z)) 
&= \rho_{\tau_1}(\rho_{\tau_2}(Z)\1_{\{\tau_1\leq \tau_2\}}+\rho_{\tau_2}(Z)\1_{\{\tau_2<\tau_1\}}) \\ 
& = \rho_{\tau_1}(\rho_{\tau_1\vee \tau_2}(Z)\1_{\{\tau_1\leq \tau_2\}})
+ \rho_{\tau_2}(Z)\1_{\{\tau_2<\tau_1\}} = \rho_{\tau_1 \wedge \tau_2}(Z).
\end{align*}
The second claim follows by taking $s=0$, $t \to \infty$ and appealing to the a.s.\ finiteness of $\tau_{y,z}$ and the continuity of $\varrho$.
\end{proof}

\label{subsec:assum}

We will use the following version of the optional sampling theorem (a related result in the case of  positively homogeneous nonlinear expectation appears in \cite[Th.\ 4.10]{Nutz_Soner}):

\begin{proposition}\label{pro:optional_stopping} 
If $M$ is a c\`adl\`ag $\varrho$-martingale then for all $\tau \in \mathscr{T}$ we have $\rho(M_\tau) = \rho(M_0)$.
\end{proposition}

\begin{proof}
Supposing first that time is discrete, we show that $\rho_n(M_{\tau \wedge (n+1)}) = M_{\tau \wedge n}$. Indeed, since $M_{\tau \wedge n}$ is $\mathcal{F}_n$-measurable and $\{\tau>n\}\in \mathcal{F}_n$ we have
\begin{multline*}
\rho_n(M_{\tau \wedge (n+1)}) - M_{\tau \wedge n}
= \rho_n(M_{\tau \wedge (n+1)} - M_{\tau \wedge n}) \\
= \rho_n((M_{\tau \wedge (n+1)} - M_{\tau \wedge n})\1_{\{\tau>n\}})
= \1_{\{\tau>n\}} \rho_n(M_{n+1} - M_{n})
= 0.
\end{multline*}
It follows by time consistency that 
\begin{equation*} 
\rho(M_{\tau \wedge (n+1)})
= \rho(M_{\tau \wedge n})
= \cdots = \rho(M_0).
\end{equation*}
Taking $n\to \infty$ it follows by the continuity of $\rho$ that $\rho(M_\tau) = \rho(M_0)$.

In continuous time there exists a sequence $(\tau_n)_{n \geq 1}$ of stopping times attaining countably many values and decreasing almost surely to $\tau$ \cite[Th.\ 77.1, p.\ 188]{Rogers_Williams_1}. Taking the limit in $\rho(M_{\tau_n})=\rho(M_0)$ and using the fact that $M$ is c\`adl\`ag and $\rho$ is continuous, we obtain the desired result.
\end{proof}

\begin{lemma}[Properties of $\mathcal{H}$]\label{lem:hprops} Writing $h = h_{\beta,\gamma}^{y,z}$, for $0 \leq y \leq x \leq z \leq 1$ and $\beta, \gamma \in [0, \bar g] 
$ we have:
    \begin{enumerate}
        \item{(Continuity in $x$)} \label{ph0} the function $x \mapsto h^{y,z}_{\beta,\gamma}(x)$ is continuous on $[y,z]$, 
        \item (Continuity in $(\beta,\gamma)$) the function $(\beta,\gamma) \mapsto h_{\beta,\gamma}^{y,z}(x)$ is continuous, \label{ph1}
        \item (Maximum principle) \label{ph2} 
        \[
        \beta \wedge \gamma \leq h(x) \leq \beta \vee \gamma, \qquad x \in [y,z],
        \]
        \item (Monotonicity) \label{ph3} 
        if $\beta<\gamma$ then $h$ is nondecreasing and if $\beta>\gamma$ then $h$ is nonincreasing,
        \item (Translation invariance) \label{ph4} 
        we have 
        \[
        h + \alpha = h_{\alpha + \beta,\alpha + \gamma}^{y,z},
        \]
        \item (Monotone derivatives) \label{ph5} suppose that the functions $h$ and $\tilde h = h_{\tilde \beta,\gamma}^{y,z} \in \mathcal{H}$ both have left derivatives at $z$. Then if $\tilde \beta > \beta$ we have 
        \[
        h'^{-}(z) \geq \tilde h'^{-}(z),
        \]
        where $^{'-}$ denotes the left derivative,
        \item (Bounding) \label{bounding} For each pair $y,u \in [0,1)$ with $y < u$ there exists $\delta > 0$ such that 
        \begin{align*}\label{lim_of_h_0_gamma}
        h^{y,z}_{0,\bar g+1}(u') > \bar g, \qquad \forall \; u' \in [u,1),
        \end{align*}
        where $z = \min\{u'+\delta,1\}$. Similarly, for each pair $u,z \in (0,1]$ with $u < z$ there exists $\delta > 0$ such that 
        \begin{align*} 
        h^{y,z}_{\bar g+1,0}(u') > \bar g, \qquad \forall \; u' \in (0,u],        
        \end{align*}
where $y = \max\{u'-\delta,0\}$. 
        \item (Scaling) in the case when the underlying diffusion $Y^x$ is standard Brownian motion, \label{ph6} taking $h := h_{0,\gamma}^{0,z}$ and $\hat h := h_{0,\gamma}^{0,\hat z}$ where $\hat z \in [z,1]$, for each $\delta \in (0, z)$ we have
        \[
        h(z - \delta) \leq \hat h(\hat z - \delta),
        \]
        and thus if $h$ and $\hat h$ are differentiable then $h'^{-}(z) \geq \hat h'^{-}(\hat z)$. Conversely, if $h := h_{\beta,0}^{0,z}$ and $\hat h := h_{\beta,0}^{0,\hat z}$ where 
        $\hat z \in [z,1]$, and if $h$ and $\hat h$ are differentiable then $h'^{-}(z) \leq \hat h'^{-}(\hat z)$.
    \end{enumerate}
\end{lemma}

\begin{proof}
For part \ref{ph0}, we will establish continuity at each $x \in (0,1)$. One-sided continuity at $0$ and $1$ can then be proved analogously by considering the underlying diffusion $Y^x$ instead of $X^x$ (cf. Section \ref{sec:setting}), since $Y$ is not absorbed at these boundary points.
Recalling that we write $S$ for the scale function of $X^x$, 
the process $\left(M_t\right)_{t\geq 0} := \left(S(Y^x_t)\right)_{t\geq 0}$ is a continuous local martingale
and by the Dubins-Schwarz Theorem \cite[Th.\ 34.1]{Rogers_Williams_2} there exists a Brownian motion $W$ such that
$M_t = W_{\langle M \rangle_t}$, 
where $\langle M \rangle$ denotes the quadratic variation of $M$.

Recall from Section \ref{sec:setting} that, for each $t>0$ and almost every $\omega \in \Omega$, the path $X^x_{0,t} := (X^x_s(\omega))_{0 \leq s \leq t}$ depends continuously on $x$ in the uniform topology. Also, by the Law of the Iterated Logarithm at $t=0$ (see e.g. \cite[p.\ 97]{Peskir2006}), for almost all $\omega\in\Omega$ there exists $\delta>0$ such that for all $t \in [0,\delta]$ we have 
\begin{align}
    \min(W_{0,\delta}) < S(x) < \max(W_{0,\delta}) < S(z), \text{ giving }       
    \min(X_{0,\delta'}^x) < x < \max(X_{0,\delta'}^x) < z,  
\end{align}
where $\langle M \rangle_{\delta'} = \delta$. Thus almost surely, for all sufficiently small $\epsilon_n > 0$ we have $\1_{\{\tau_x^{x+\epsilon_n} < \tau_z^{x+\epsilon_n}\}} = 1$.

Take $\epsilon_n \downarrow 0$. Then for $0 \leq y \leq x \leq z \leq 1$ and $x + \epsilon_n \in [y,z]$, we have from the strong Markov property for $\varrho$ that
\begin{align*}
\rho_{\tau_{x,z}^{x+\epsilon_n}}\left(h\left(X_{\tau_{y,z}}^{x+\epsilon_n}\right)\right) 
&= \rho_{\tau_{x,z}^{x+\epsilon_n}}\left(h\left(X_{\tau_{y,z}}^{X_{\tau_{x,z}}^{x+\epsilon_n}}\right)\circ \theta_{\tau_{x,z}}\right) 
= \rho\left(h\left(X_{\tau_{y,z}}^{{X_{\tau_{x,z}}^{x+\epsilon_n}}}\right)\right) \\
&= \1_{\{\tau_x^{x+\epsilon_n} < \tau_z^{x+\epsilon_n}\}} \rho\left(h\left(X_{\tau_{y,z}}^{x}\right)\right)
+ \1_{\{\tau_x^{x+\epsilon_n} > \tau_z^{x+\epsilon_n}\}} \rho\left(h\left(X_{\tau_{y,z}}^{z}\right)\right),
\end{align*}
and time consistency gives
\begin{align*}
\rho\left(h\left(X_{\tau_{y,z}}^{x+\epsilon_n}\right)\right) 
&= \rho\left(\1_{\{\tau_x^{x+\epsilon_n} < \tau_z^{x+\epsilon_n}\}} \rho\left(h\left(X_{\tau_{y,z}}^{x}\right)\right)
+ \1_{\{\tau_x^{x+\epsilon_n} > \tau_z^{x+\epsilon_n}\}} \rho\left(h\left(X_{\tau_{y,z}}^{z}\right)\right)\right). 
\end{align*}
We conclude from the continuity of $\rho$ that 
\begin{align*}
h^{y,z}_{\beta,\gamma}(x + \epsilon_n) = \rho(h(X_{\tau_{y,z}}^{x+\epsilon_n})) \to
 \rho(h(X_{\tau_{y,z}}^{x}))) = \rho(h(X_{\tau_{y,z}}^{x})) = h^{y,z}_{\beta,\gamma}(x),
\end{align*}
establishing the right continuity of $x \mapsto h^{y,z}_{\beta,\gamma}(x)$ on $[y,z]$, while left continuity is established via symmetric arguments.

For part \ref{ph1} take a sequence $(\beta_n, \gamma_n) \to (\beta, \gamma)$, define the functions $f_n(u):=\beta_n + \frac{u - y}{z - y} (\gamma_n - \beta_n)$, $f(u):=\beta + \frac{u - y}{z - y} (\gamma - \beta)$ and set $Y_n := f_n(X_{\tau_{y,z}}^x)$. Then a.s.\ we have $Y_n \to Y:= f(X_{\tau_{y,z}}^x)$, so the continuity of $\varrho$ gives
\[
h_{\beta_n,\gamma_n}^{y,z}(x) = \rho(Y_n) \to \rho(Y) = h_{\beta,\gamma}^{y,z}(x). 
\]

Part \ref{ph2} follows from the monotonicity of $\varrho$.

For part \ref{ph3}, note that if $\beta<\gamma$ and  $y\leq x'\leq x \leq z$ then we have from part \ref{ph2} that $\beta \leq h(x)$, so the $\varrho$-martingale property (Lemma \ref{lem:mgale}) and monotonicity of $\varrho$ give
$$h(x') = \rho(h(X_{\tau_{y,x}}^{x'})) \leq h(y) \vee h(x) = \beta \vee h(x) = h(x).$$
Part \ref{ph4} follows from the translation invariance of $\varrho$ and part \ref{ph5} follows from the monotonicity of $\varrho$, since for $v \in [y,z]$ we have
\[
\frac{h(z)-h(v)}{z-v} = \frac{\tilde{h}(z)-h(v)}{z-v} \geq \frac{\tilde h(z)- \tilde h(v)}{z-v}. 
\]
    
    For part \ref{bounding}, suppose that there exist $y,u \in [0,1)$ with $y < u$ such that
    \begin{equation}
    \label{eq:sss}
    M:=\sup_{\delta > 0}\inf_{u' \in [u,1)} h^{y,\min\{u'+\delta,1\}}_{0,\bar g + 1
    }(u') \leq \bar g. 
    \end{equation}
    In the definition of $\Delta$-equicontinuity (Definition \ref{def:hcal}),
    take $\Delta = u-y$, let $\epsilon \in (0,1)$, and let $\delta > 0$ be the resulting constant. Then for each $u'\in [u,1)$ we have by $\Delta$-equicontinuity that
    \begin{align*}        
    h^{y,\min\{u'+\delta,1\}}_{0,\bar g+1
    }(u') \geq 
    h^{y,\min\{u'+\delta,1\}}_{0,\bar g+1
    }(\min\{u'+\delta,1\}) - \epsilon
    = \bar g+1
    - \epsilon, 
    \end{align*}  
    and taking the infimum over $u'\in [u,1)$ and supremum over $\delta > 0$ gives $M \geq \bar g + 1$, contradicting \eqref{eq:sss}. The remaining part is analogous.

For part \ref{ph6} we have $\1_{\left\{\tau_z^{z - \delta} < \tau_0^{z - \delta}\right\}} =  \1_{\left\{\tau_{\hat z}^{\hat z - \delta} < \tau_{\hat z - z}^{\hat z - \delta}\right\}}$ by the spatial homogeneity of the Wiener process (and the fact that $z - \delta, \hat z - \delta \in (0,1)$ and $z, \hat z, \hat z - z \in [0,1]$).
Then if $h := h_{0,\gamma}^{0,z}$ and $\hat h := h_{0,\gamma}^{0,\hat z}$,   
    \begin{align*}
    h(z - \delta) &= \rho(\gamma \1_{\left\{\tau_z^{z - \delta} < \tau_0^{z - \delta}\right\}}) = \rho (\gamma \1_{\left\{\tau_{\hat z}^{\hat z - \delta} < \tau_{\hat z - z}^{\hat z - \delta}\right\}}) \leq \rho (\gamma \1_{\left\{\tau_{\hat z}^{\hat z - \delta} < \tau_0^{\hat z - \delta}\right\}}) = \hat h(\hat z - \delta),
    \end{align*}
    where the inequality follows from the monotonicity of $\varrho$ and the fact that 
    $\1_{\left\{\tau_{\hat z}^{\hat z - \delta} < \tau_{\hat z - z}^{\hat z - \delta}\right\}} \leq \1_{\left\{\tau_{\hat z}^{\hat z - \delta} < \tau_0^{\hat z - \delta}\right\}}$, $\PP$-a.s.. 
\end{proof}

\begin{remark}\label{rem:disc}
    In financial and economic applications one may wish to account for the time value of money  in problem \eqref{optimal_stoping_intro} via exponential discounting, by setting $V^\delta(x):=\sup_{\tau\in\mathscr{T}} \rho(e^{-\delta \tau}g(X^x_\tau))$. It was shown in 
    \cite{Dayanik_Karatzas} that under linear expectation, $V^\delta$ is the envelope of  discounted potentials of the form 
    $\E\left(e^{-\delta \tau_{y,z}}\left(\beta \1_{\{\tau_y < \tau_z\}} + \gamma \1_{\{\tau_y > \tau_z\}}\right)\right)$.
    Nonetheless the linear expectation has the property of positive homogeneity (cf. Definition \ref{def:hbasics}), which would enable discounting to be included in Lemma \ref{lem:mgale}. In contrast our nonlinear setting is not positively homogeneous (cf. Section \ref{subsec:risk_mappings}), hence there is no obvious discounted analogue of Lemma \ref{lem:mgale}. More precisely, if we instead define $M_t = e^{-\alpha t\wedge\tau_{y,z}}h_{\beta,\gamma}^{y,z}(X^x_{t\wedge\tau_{y,z}})$ in Lemma \ref{lem:mgale} then its first line reads
\begin{align*}
\rho_s(M_t) 
&= \rho_s \left(e^{-\alpha t\wedge\tau_{y,z}} h\left(X^x_{t\wedge \tau_{y,z}}\right) \right) 
= \rho_s \left(e^{-\alpha t\wedge\tau_{y,z}} \rho\left(e^{-\alpha \tau_{y,z}}g\left(X^{X^x_{t\wedge \tau_{y,z}}}_{\tau_{y,z}}\right)\right) \right).
\end{align*}
Without a suitable multiplicative property, the term $e^{-\alpha t\wedge\tau_{y,z}}$ cannot be brought inside the inner expectation and, thus, the strong Markov property cannot be applied in order to proceed with the proof. We thus leave consideration of discounting as an interesting open question.
\end{remark}

\begin{lemma}\label{lem:diffislec}
    If the underlying diffusion $Y$ is standard Brownian motion and if $\mathcal{H}$ is differentiable then $\mathcal{H}$ is $\Delta$-equicontinuous.
\end{lemma}

\begin{proof}
    Take $h = h^{y,z}_{\beta,\gamma}$ 
    with $z-y > \Delta$ and 
$\beta, \gamma \in [0,\bar g + 1]$. By the spatial homogeneity of the Wiener process we may suppose without loss of generality that $y=0$. 
    We will show that there exists a constant $M>0$ independent of $h$ such that $|h'(u)| < M$ for all $u \in (0,z)$, since this is sufficient to establish the required result. For this, note that for all such $u$, by the $\varrho$-martingale property and parts \ref{ph4}--\ref{ph6} of Lemma \ref{lem:hprops}, setting $m:= \bar g$ we have
    \begin{align*}
    h'(u) &= h'^{-}(u) = (h_{h(0),h(u)}^{0,u})^{'-}(u) = (m - h(u) + h_{h(0),h(u)}^{0,u})^{'-}(u) 
     \\ & = (h_{m - h(u) + h(0),m}^{0,u})^{'-}(u) 
     \leq  (h_{0,m}^{0,u})^{'-}(u) \leq (h_{0,m}^{0,\Delta})^{'-}(\Delta),  
    \end{align*}
    and similarly $h'(u) \geq  (h_{m,0}^{0,\Delta})^{'-}(\Delta)$. 
    Thus $\abs{h'(u)} \leq M:=(h_{0,m}^{0,\Delta})^{'-}(\Delta) \vee (-(h_{m,0}^{0,\Delta})^{'-}(\Delta))$, a bound which is independent both of $h$ and of the particular choice of $u$.
\end{proof}

\subsection{Representation of the value function}\label{sec:repvf}

Next we use these properties of $\mathcal{H}$ to establish the structure of the value function (cf. Figure \ref{fig:1}). 
\begin{lemma}\label{lem:wusc}
    The \bl{lower envelope} $w$ of \eqref{eq:repV} satisfies $w(0)-g(0)=w(1)-g(1)=0$.
\end{lemma}

\begin{proof}
If $g(0)=\bar g$ then trivially $w(0)=g(0)=\bar g$. If $g(0) < \bar g$ then let $\epsilon \in (0, \bar g - g(0))$
and let $\delta > 0$ be such that $g < g(0) + \epsilon$ on $[0,\delta]$. Then by the monotonicity of $\varrho$ the function $h = h^{0,\delta}_{g(0) + \epsilon, \bar g} \in \mathcal{H}$ satisfies $h \geq g$, giving $w(0) \leq h(0) = g(0) + \epsilon$ and completing the claim.
\end{proof}

\begin{theorem}\label{pro:representation_4param}
Under Assumption \ref{continuity_assumption} we have $V = w$.
\end{theorem}
\begin{remark}\label{rem:wrl}
The proof of Theorem \ref{pro:representation_4param} proceeds by first localising the problem to a minimal interval $[x^-,x^+] \ni x$ such that $V=g$ at $x^\pm$ (Proposition \ref{lem:wgV}), and obtaining the corresponding solution $h^\circ$ (Lemma \ref{lem:hcirc}). Proposition \ref{cor:vvc} then uses the strong Markov property for $\varrho$ (Lemma \ref{lem:tau3}) to  show that the local and global solutions coincide.

\end{remark}

\begin{lemma}\label{lem:parti}
    For $x\in (0,1)$ such that $g(x) = w(x)$, we have $w(x)=g(x)=V(x)$.
\end{lemma}

\begin{proof} 
Let $\tau \in \mathscr{T}$. For each fixed $\Delta>0$, by \eqref{eq:repV} there exists $h = h_{\beta,\gamma}^{y,z} \in \mathcal{H}$ with $h \geq g$ and $w(x) \geq h(x) - \Delta$. 
In particular, $x \in [y,z]$ by finiteness. 
By the $\varrho$-martingale property and optional sampling (Proposition \ref{pro:optional_stopping}) we have 
\begin{align}\label{eq:nostep}
    g(x) &= w(x) \geq h(x) - \Delta = \rho(h(X_{\tau \wedge \tau_{y,z}}^x)) - \Delta \geq \rho(g(X_\tau^x)) - \Delta,
\end{align}
where the last inequality can be verified case-by-case, as follows. If $\tau < \tau_{y,z}$ we have $h(X_{\tau\wedge\tau_{y,z}}^x) = h(X_\tau^x) \geq g(X_\tau^x)$. If $\tau_{y,z} \leq \tau$ with $y>0$ and $z=1$, then in the case $\tau_{y,z}=\tau_y$ we have $h(X_{\tau\wedge\tau_{y,z}}^x) = h(y) = \bar g \geq g(X_\tau^x)$ by the construction of $\mathcal{H}$ in Definition \ref{def:hcal}, while in the case $\tau_{y,z}=\tau_z$ we have $h(X_{\tau\wedge\tau_{y,z}}^x) = h(1) \geq g(1) = g(X_\tau^x)$ due to absorption at $1$. We conclude from \eqref{eq:nostep} that $g(x) \geq V(x)$. The reverse inequality is obvious and we have $g(x)=w(x)=V(x)$.
\end{proof}

Now, fix $x\in (0,1)$ such that $w(x) > g(x)$. By Lemma \ref{lem:wusc} we may define 
\begin{align}
x^- &= \sup \{u \in [0,x) : w(u) = g(u) \}, \label{x_minus}\\
x^+ &= \inf \{u \in (x,1] : w(u) = g(u) \}, \label{x_plus}\\
\tau^* &= \tau_{x^-,x^+} \leq \tau_{0,1}. \label{eq:tst}
\end{align}

\begin{proposition}\label{lem:wgV}
    $w(x^\pm) = g(x^\pm) = V(x^\pm)$.
\end{proposition}

\begin{proof}
Suppose for a contradiction that $w(x^-) - g(x^-) = \Delta > 0$. Then if $x^-=0$, the required contradiction follows from Lemma \ref{lem:wusc}. If $x^->0$ then by construction, for each $\epsilon > 0$ there exists $u \in (x^- - \epsilon, x^-]$ and $h^\epsilon \in \mathcal{H}$ with $w(u) = g(u)$, $h^\epsilon \geq g$ and $h^\epsilon(u) \leq w(u) + \frac{\Delta}{3}$, and by the definition of $w$ we have $h^\epsilon(x^-) \geq w(x^-) = g(x^-) + \Delta$. By continuity of $g$, for sufficiently small $\epsilon$ we have $w(u) = g(u)< g(x^-) + \frac{\Delta}{3}$ and hence 
\begin{align}\label{eq:Deltaequi}
h^\epsilon(x^-) - h^\epsilon(u) \geq g(x^-) + \Delta - h^\epsilon(u)
\geq g(x^-) - w(u) + \frac{2\Delta}{3} > \frac{\Delta}{3}.
\end{align}
However each $h^\epsilon$ lies in the set $\mathcal{H}$ (Definition \eqref{def:hcal}), so may be written $h^\epsilon = h^{y,z}_{\beta,\gamma}$ with $[y,z] \ni x$ and either $y=0$ or $z=1$. Thus by the $\Delta-$equicontinuity of $\mathcal{H}$, the functions $h^\epsilon$ appearing in \eqref{eq:Deltaequi} are equicontinuous, which is a contradiction. We conclude that $w(x^-) = g(x^-) = V(x^-)$ (cf. Lemma \ref{lem:parti}) and similarly $w(x^+) = g(x^+) = V(x^+)$.
\end{proof}

\begin{lemma}\label{lem:hcirc}
If $g(x) < w(x)$ then there exists $h^\circ \in \mathcal{H}$ with $h^\circ(x^-)=g(x^-)$, $h^\circ(x^+)=g(x^+)$, and $h^\circ \geq g$.  
\end{lemma}
\begin{proof} 
We exclude the trivial case $g(x^-) = g(x^+) = \bar g$, when $h^\circ$ may simply be taken as the constant function $h^{0,1}_{\bar g,\bar g} \in \mathcal{H}$. 

The idea of the proof is to extend the function $h^{x^-,x^+}_{g(x^-),g(x^+)}$ to a function $h^\circ \in \mathcal{H}$ in such a way that $h^\circ \geq g$.  There will be four intermediate steps:
\begin{enumerate}
    \item \label{s1} Extend the function $h_0 := h^{x^-,x^+}_{g(x^-),g(x^+)}$ to the right, obtaining a function 
    $h_{n'} := h^{x^-,Z}_{g(x^-),\Gamma}$ 
    satisfying either $Z=1$ or $\Gamma \geq \bar g$,
    \item \label{s2} extend the function $h_{n'}$ to the left, obtaining a function $\bar{h}_N:=h^{W,Z}_{B,\Gamma}$ satisfying either $W=0$ or $B \geq \bar g$,
    \item \label{s3} verify that $\bar{h}_N$ dominates $g$ on $[0,x^-] \cup [x^+,1]$,
    \item \label{s4} restrict the function $\bar{h}_N$ to obtain an element $h^\circ = h^{W',Z'}_{B',\Gamma'}$ with $h^\circ \in \mathcal{H}$, and finally verify that $h^\circ \geq g$.
\end{enumerate}

{\it Step \ref{s1}.} We first extend $h_0 := h^{x^-,x^+}_{g(x^-),g(x^+)}$ to the right. If $g(x^+) = \bar g$ then set $\Gamma := \bar g$ and $Z = x^+$, and step \ref{s1} is complete. Otherwise we have $g(x^+) < \bar g$. Note from Proposition \ref{lem:wgV} that $g(x) < w(x)$ implies $x^- < x^+$. Then take $y=x^-$, $u=x^+$ in the bounding property (part \ref{bounding} of Lemma \ref{lem:hprops}) to obtain the corresponding $\delta$. Set $z_0 := x^+$  and $z_1 = \min\left\{x^+ + \delta,1\right\}$. 

Suppose for a contradiction that $\tilde{h}(x^+) > g(x^+)$, where $\tilde{h} := h^{x^-,z_1}_{g(x^-),0}$. Then since $w(x^+) = g(x^+) < \tilde{h}(x^+)$, by the definition of $w$ there exists $h \in \mathcal{H}$ with 
\begin{align*}
h \geq g, \qquad
h(x^+)  < \tilde{h} (x^+), \qquad h(x^-)  \geq g(x^-) = \tilde{h} (x^-), 
\end{align*}
so that by the continuity of $\mathcal{H}$ we have $h(u)=\tilde{h}(u)$ for some $u \in \xpm$. Next note that $h(z_1)\geq 0 = \tilde{h}(z_1)$, so by the continuity of $\mathcal{H}$ we have $h(v) = \tilde{h}(v)$ for some $v \in (x^+,z_1]$. Then since $h$ and $\tilde h$ coincide at $u \leq x^+$ and at $v > x^+$, the $\varrho$-martingale property gives $h(x^+) = \tilde{h}(x^+)$, a contradiction. We conclude that $\tilde{h}(x^+)=h^{x^-,z_1}_{g(x^-),0} (x^+) \leq g(x^+)$. 

Then recalling the constant $\delta$ obtained above from the bounding property with $y=x^-$ and $u=x^+$, we set $u'=x^+$ to obtain $h^{x^-,z_1}_{0,\bar g+1} (x^+) > \bar g$. Further, the monotonicity of $\varrho$ gives $h^{x^-,z_1}_{g(x^-),\bar g + 1} (x^+) > \bar g$. Then, from `continuity in $(\beta,\gamma)$' (part \ref{ph1} of Lemma \ref{lem:hprops}) and the intermediate value theorem, there exists $\gamma \in [0,\bar g+1)$ 
such that $h^{x^-,z_1}_{g(x^-),\gamma} (x^+) = g(x^+)$.
Set
\begin{align*}
\gamma_1 &:= \sup\{\gamma \geq 0: h^{x^-,z_1}_{g(x^-),\gamma} (x^+) = g(x^+)\}, \\   
h_1 &:= h^{x^-,z_1}_{g(x^-),\gamma_1},
\end{align*}
noting by continuity that $h_1(x^+) = g(x^+)$.
If either $z_1=1$ or $\gamma_1 \geq \bar g$
then set $\Gamma := \gamma_1$ and $Z = z_1$, and step \ref{s1} is complete. Otherwise, if $z_1 < 1$ and $\gamma_1 < \bar g$ 
then set $z_2 = \min\left\{x^+ + 2 \delta,1\right\}$. 
As before we have $h^{x^-,z_2}_{g(x^-),0} (x^+) \leq g(x^+)$. The definition of $\gamma_1$, $\varrho$-martingale property and monotonicity of $\varrho$ then give $h^{x^-,z_2}_{g(x^-),0} (z_1) \leq \gamma_1$. 
Also taking $u'=z_1$ in the bounding property gives $h^{x^-,z_2}_{0,\bar g+1} (z_1) > \bar g$ and the monotonicity of $\varrho$ gives $h^{x^-,z_2}_{g(x^-),\bar g + 1} (z_1) > \bar g$ and, as before, we may write
\begin{align*}
\gamma_2 &:= \sup\{\gamma \geq 0: h^{x^-,z_2}_{g(x^-),\gamma} (z_1) = \gamma_1\}, \\   
h_2 &:= h^{x^-,z_2}_{g(x^-),\gamma_2}.
\end{align*}
Proceeding similarly, for some $n' \in \N$ we have either $z_{n'} = 1$ or $\gamma_{n'} \geq \bar g$, and set $\Gamma = \gamma_{n'}$ and $Z = z_{n'}$, and step \ref{s1} is complete. 

{\it Step \ref{s2}.} In Step \ref{s1} the function $h_0 := h^{x^-,x^+}_{g(x^-),g(x^+)}$ was extended to the right, obtaining a function $h_{n'} := h^{x^-,Z}_{g(x^-),\Gamma}$ satisfying either $Z=1$ or $\Gamma \geq \bar g$. Indeed it can be verified by construction and the $\varrho$-martingale property that for each $n \geq 1$ we have $h_n = h_{n-1}$ on $[x^-, z_{n-1}]$, giving 
\begin{align}\label{eq:hnext}
h_n(v) = h_0(v) = \rho(g(X_{\tau^*}^v)), \qquad \forall \; v \in \xpm,    
\end{align}
where $\tau^*$ was defined in \eqref{eq:tst} as $\tau_{x^-,x^+}$.

Next we extend $h_{n'} := h^{x^-,Z}_{g(x^-),\Gamma}$ to the left. Setting $w_1 = \max\left\{0,x^- - \delta\right\}$ (for a suitable $\delta > 0$) we may define
\begin{align}\label{eq:oppit}
\beta_1 &= \sup\{ \beta \geq 0: h^{w_1,x^+}_{\beta,g(x^+)}(x^-) = g(x^-)\},     \\
\bar h_1 &= h^{w_1,Z}_{\beta_1,\Gamma}. \nonumber
\end{align}
By an iterative extension analogous to the previous one, we obtain $B \geq 0$ and $W \in [0,x^-]$ such that either $B \geq \bar g$ or $W = 0$.
Let $N$ be the index at which the iteration stops. 
As in \eqref{eq:hnext} we have $\bar h_N = \bar h_1 = h_0$ on $\xpm$. 

{\it Step \ref{s3}.} In Step \ref{s2} the function $h_{n'}$ was extended to the left, obtaining a function $\bar{h}_N:=h^{W,Z}_{B,\Gamma}$ satisfying either $W=0$ or $B \geq \bar g$. Next we verify that $\bar{h}_N \geq g$ on $[0,x^-]$. 
Suppose for a contradiction that $\bar{h}_N(w) < g(w)$ 
for some $w \in [0,x^-]$. Since then $\bar{h}_N(w) < \infty$, 
we have $w_N \leq w$ and there exists $\tilde n$ with $1 \leq \tilde n \leq N$ such that $w \in [w_{\tilde n},w_{\tilde n-1}]$. 
Defining the function $h_w := h^{w,x^+}_{g(w),g(x^+)}$, we have $h_w(w) = g(w) \leq \bar g$ and so the above procedure may be applied (with the same value for $\delta$) to extend the function $h_w$ to the left, yielding $h^*_w := h^{w',x^+}_{\gamma',g(x^+)}$ where $w' = \max\{w - \delta,0\} \leq w_{\tilde n }$. By the definition of $V$ and of $h_w$, and recalling Proposition \ref{lem:wgV}, suboptimality gives $h_w(x^-) \leq V(x^-)= g(x^-)$. On the other hand, the $\varrho$-martingale property and monotonicity of $\varrho$ give $h_w(x^-) \geq \bar{h}_N(x^-) = g(x^-)$, so $h_w(x^-) = g(x^-)$. The suprema taken during the extension process (cf. \eqref{eq:oppit}) therefore give $h^*_w(w_k) \leq \bar{h}_N(w_k)$ for $k =1,2,\ldots, \tilde n$. 
The $\varrho$-martingale property then gives $g(w) = h^*_w(w) \leq \rho\left(\bar{h}_N\left(X_{\tau_{w_{\tilde n-1},w_{\tilde n}}}^w \right)\right)=\bar{h}_N(w)$, which is the required contradiction. An analogous argument establishes that $\bar{h}_N \geq g$ on $[x^+,1]$. 

{\it Step \ref{s4}.} In Step \ref{s3} it was verified that the function $\bar{h}_N = h^{W,Z}_{B,\Gamma}$ dominates $g$ on $[0,x^-] \cup [x^+,1]$. Since $\bar{h}_N$ does not necessarily satisfy the definition \eqref{eq:defcalh} of an element of $\mathcal{H}$, in this step we first identify a `restriction' $h^\circ$ of the function $\bar{h}_N$ with $h^\circ \in \mathcal{H}$. This is achieved by modifying the constants $W$, $Z$, $B$ and $\Gamma$, as follows. If $\Gamma \geq \bar g$, by continuity in $x$ (Part \ref{ph0} of Lemma \ref{lem:hprops}) there exists $Z' \in (z_{n'-1},Z]$ with $h_{n'}(Z') = \bar g$, and we then define $\Gamma'=\bar g$. Otherwise, if $\Gamma < \bar g$ then by the above procedure we must have $Z=1$; set $Z'=Z=1$ and $\Gamma' = \Gamma$. Similarly, if $B \geq \bar g$ there exists $W' \in [W,w_{N-1})$ with $\bar{h}_{N}(W') = \bar g$, and we then define $B'=\bar g$; otherwise, if $B < \bar g$ then $W=0$ and we set $W'=W=0$ and $B' = B$.

Writing $h^\circ := h^{W',Z'}_{B', \Gamma'}$, we observe that either $W=0$ or $Z=1$. (This follows since if both $W'>0$ and $Z'<1$ then $B = \Gamma = \bar g$, and the monotonicity of $\varrho$ then forces $g(x^-) = g(x^+) = \bar g$, which was excluded at the beginning of the proof.) It is now immediate from Definition \ref{def:hcal} that the function $h^\circ := h^{W',Z'}_{B', \Gamma'}$ is an element of $\mathcal{H}$.  It is also immediate from Definition \ref{def:hcal} that 
\[
h^\circ(x) = 
\begin{cases}
    \bar{h}_N(x), & x \in [W',Z'], \\
    \infty, & x \in [0,W') \cup (Z',1],
\end{cases}
\]
so the above results give $h^\circ \geq \bar{h}_N \geq g$ on $[0,x^-] \cup [x^+,1]$. 

To complete the proof it remains only to establish that $h^\circ \geq g$ on $\xpm$. For this, let
\begin{align}\label{eq:defep1}
\epsilon = \sup\{g(u)-h^\circ(u): u \in \xpm\},
\end{align}
and suppose for a contradiction that $\epsilon > 0$. Then the function $h^\circ + \epsilon$ dominates $g$ on $\xpm$ and thus on $[0,1]$. Recalling the translation invariance property of $\mathcal{H}$ (part \ref{ph4} of Lemma \ref{lem:hprops}), the function $h^\circ + \epsilon = h^{W',Z'}_{B'+\epsilon, \Gamma'+\epsilon}$ can be `restricted' to obtain an element $h^\circ_{\epsilon} = h^{W'',Z''}_{B'', \Gamma''} \in \mathcal{H}$ which also dominates $g$ on $[0,1]$. Then by  continuity the supremum in \eqref{eq:defep1} is attained on $\xpm$ and by the construction of $h^\circ_\epsilon$ it must be attained at some $u \in (x^-,x^+)$. 
Then we have 
\[
g(u) \leq w(u) \leq h^\circ_{\epsilon}(u) = 
h^\circ(u)
+ \epsilon = g(u),
\] thus $g(u) = 
w(u)$, contradicting \eqref{x_minus}--\eqref{x_plus}. We conclude that $h^\circ \geq g$ on $\xpm$ and thus on $[0,1]$.
\end{proof}

\begin{lemma} \label{lem:tau3} Let $\tau_1, \tau_2 \in \mathscr{T}$
with $0 \leq \tau_1 \leq \tau_2$ a.s..
Then there exists a (random) stopping time $\tau_3 \in \mathscr{T}$ such that almost surely we have
\begin{align*}
 \rho_{\tau_1}(g(X_{\tau_2}^x)) 
&= \begin{cases}
\rho\left(g\left(X_{\tau_3}^{X_{\tau_1}^x}\right)\right), & \{\tau_1 < \tau_2\},   \\
g(X_{\tau_1}^x), & \{\tau_1 = \tau_2\}.
\end{cases}
\end{align*}
\end{lemma}

\begin{proof}
We divide the proof into two steps.

{\it Step (i).} Fix $\omega \in \Omega$ and consider the map $G^\omega \colon \Omega \to \Omega$ given by $\bar \omega \mapsto G^{\omega}(\bar \omega)$, which constructs a sample path $G^{\omega}(\bar \omega)$ by (a) following $\omega$ from time 0 until time $\tau_1(\omega)$, then (b) following the increments of $\bar \omega$ from time 0 onwards. That is,
\begin{align*}
    G^\omega(\bar \omega)(t) &= 
\begin{cases}
    \omega(t), & t \leq \tau_1(\omega),\\
    \omega(\tau_1(\omega)) + \bar \omega(t - \tau_1(\omega)) - \bar \omega(0), & t > \tau_1(\omega).
\end{cases}
\end{align*}
For any stopping time $\tau \in \mathscr{T}$ consider the random time $\tau^\omega(\bar \omega)$ obtained by applying the function $\tau$ to this sample path:
\begin{align*}
    \bar \omega \mapsto \tau^\omega(\bar \omega) &:= \tau(\tilde \omega), \qquad \text{ where } \qquad \tilde \omega := G^{\omega}(\bar \omega).
\end{align*}
Then since $G^\omega(\theta_{\tau_1(\omega)}(\omega)) = \omega$, we have $\tau^\omega(\theta_{\tau_1(\omega)}(\omega)) = \tau(\omega)$ and so $\tau_i^\omega \circ \theta_{\tau_1(\omega)} (\omega) = \tau_i(\omega)$ for $i=1,2$. Thus 
\[
\tau_2(\omega) = \tau_1(\omega) + (\tau_2^\omega - \tau_1^\omega) \circ \theta_{\tau_1(\omega)}(\omega),
\]
and recalling that $\tau_1 \leq \tau_2$ almost surely we have 
\begin{align}
  g\left(X_{\tau_2(\omega)}^x \right) =  
  g\left(X_{\tau_2^\omega - \tau_1^\omega}^{X^x_{\tau_1(\omega)}} \right) \circ \theta_{\tau_1(\omega)}(\omega). \label{eq:elab}
\end{align}
Define an auxiliary bounded random variable $Z^{\tau_1(\omega)}$ by
\begin{align}
    Z^{\tau_1(\omega)} &\colon \Omega \to \R, \notag \\
    Z^{\tau_1(\omega)}(\bar \omega) & = g\left(X_{\tau_2^{\omega} - \tau_1^{\omega}}^{X^x_{\tau_1(\omega)}} (\bar \omega) \right). \label{eq:def-Z}
\end{align}    

Now suppose that $\bar \omega$ is an independent random draw from $(\Omega, \mathcal{F}, \PP)$. 
Putting the above together we have that almost surely
\begin{align}
  \rho_{\tau_1(\omega)}(g(X_{\tau_2(\omega)}^x)) =  \rho_{\tau_1(\omega)}(Z^{\tau_1(\omega)} 
  \circ \theta_{\tau_1(\omega)}(\omega)) 
  = \rho_{\tau_1(\tilde \omega)}(Z^{\tau_1(\omega)} 
  \circ \theta_{\tau_1(\tilde \omega)}(\tilde \omega)),
\label{eq:2paths}
\end{align}
where the first equality follows from \eqref{eq:elab}--\eqref{eq:def-Z}. For the second equality, note the identity $\tau_1(\omega) = \tau_1(\tilde \omega)$, and also that the  Brownian increments $(\omega(\tau_1(\omega)+t)-\omega(\tau_1(\omega)))_{t > 0}$ are independent of $\mathcal{F}_{\tau_1(\omega)}$ by the strong Markov property. Then since $\bar \omega$ also has Brownian increments independent of $\mathcal{F}_{\tau_1(\omega)}$, we may
replace $\theta_{\tau_1(\omega)}(\omega)$ by $\theta_{\tau_1(\tilde \omega)}(\tilde \omega)$ in \eqref{eq:2paths}. 

{\it Step (ii).} 
Next we identify $\tau_3^\omega := \tau_2^{\omega} - \tau_1^{\omega}$ with an $\mathbb{F}$-stopping time.
For this, note that the filtration $\mathbb{F}$ can be identified with the smallest filtration satisfying the usual conditions constructed from the natural filtration of $(\tilde \omega(t))_{t \geq 0}$. Thus writing $\tilde{\mathbb{F}} = (\tilde{\mathcal{F}}_t)_{t \geq 0}$ for the latter, the $\mathbb{F}$-stopping time $\tau_2$ can be identified with an 
$\tilde{\mathbb{F}}$-stopping time. Then
since $\tau_1(\tilde \omega) = \tau_1(\omega)$, for each $\bar \omega \in \Omega$ we have
\begin{align*}
\{(\tau_2^\omega - \tau_1^\omega)(\bar \omega) \leq t \} &= 
 \{\tau_2(\tilde \omega) \leq \tau_1(\omega) + t \}
\in \tilde{\mathcal{F}}_{\tau_1(\omega)+t} =: \mathcal{F}_t^{\tau_1(\omega)},
\end{align*}
so that $\tau_3^\omega$ is measurable in the filtration $\mathbb{F}^{\tau_1(\omega)} := (\mathcal{F}_t^{\tau_1(\omega)})_{t \geq 0}$. 

Next, fix $\omega$ and note that the filtration $\mathbb{F}^{\tau_1(\omega)}$
is then simply that generated by the increments of the process $(\bar \omega(t))_{t \geq 0}$. In particular $\mathbb{F}^{\tau_1(\omega)}$ is coarser than $\mathbb{F}$, so any $\mathbb{F}^{\tau_1(\omega)}$-stopping time 
can be identified with an $\mathbb{F}$-stopping time.
By the strong Markov property for Brownian motion we have $\tau_2^\omega < \infty$ a.s., giving $0 \leq \tau_3^\omega < \infty$ a.s.. We conclude that $\tau_3^\omega$ can be identified with an $\mathbb{F}$-stopping time.

Combining 
\eqref{eq:def-Z} and \eqref{eq:2paths} and then applying the strong Markov property of $\varrho$ gives almost surely
\begin{align*}
  \rho_{\tau_1(\omega)}(g(X_{\tau_2(\omega)}^x)) &= \rho_{\tau_1(\tilde \omega)}\left(g\left(X_{\tau_3^\omega}^{X^x_{\tau_1(\tilde \omega)}}\right) \circ \theta_{\tau_1(\tilde \omega)}(\tilde \omega)\right) \\
  &= \rho\left(g\left(X_{\tau_3^\omega}^{X_{\tau_1(\tilde \omega)}^x}\right)\right) = \rho\left(g\left(X_{\tau_3^\omega}^{X^x_{\tau_1(\omega)}}\right)\right),
\end{align*}
completing the proof.
 \end{proof}

\begin{proposition}\label{cor:vvc}
For $x\in [0,1]$ and $\tau^*$ defined in \eqref{eq:tst} we have
\[
V(x) = \check V(x) := \sup_{\tau \in \mathscr{T}, \tau \leq \tau^*} \rho(g(X_{\tau}^x)).
\]
\end{proposition}

\begin{proof} Let $\tau \in \mathscr{T}$. Setting $\tau_1=\tau \wedge \tau^*$ and $\tau_2=\tau$ in Lemma \ref{lem:tau3} gives that for some stopping time $\tau_3$,
\begin{align}
    \rho(g(X_\tau^x)) 
    &= \rho(\rho_{\tau \wedge \tau^*}(g(X_\tau^x))) = 
    \rho\left(\1_{\{\tau < \tau^*\}}g(X_\tau^x)
    + \1_{\{\tau \geq \tau^*\}}\rho\left(g\left(X_{\tau_3}^{X_{\tau^*}^x}\right)\right)\right) \notag\\
    & \leq \rho(\1_{\{\tau < \tau^*\}}
    g(X_{\tau}^x)
    + \1_{\{\tau \geq \tau^*\}}
    g(X_{\tau^*}^x))
    = \rho(g(X_{\tau \wedge \tau^*}^x)), \notag
\end{align}
where the first equality follows from time consistency and the inequality follows from the fact that $g(X_{\tau^*}^x)=V(X_{\tau^*}^x)$ and optimality. This gives $V \leq \check V$ and hence $V = \check V$.
\end{proof}

Thus to evaluate $V(x)$ it suffices to study $\check V$ on $\xpm$. 
For this, the above analysis may be repeated with the interval $[0,1]$ replaced by $\xpm$. Denote the resulting analogues of all objects with a check mark so that, for example, 
$$\check{w}(x)=\inf \left\{ \check{h}(x) : \check{h} \in \check{\mathcal{H}}, \check{h} \geq g \text{ on } [x^-,x^+]\right\}.$$

\begin{proof}[Proof of Theorem \ref{pro:representation_4param}]
From Lemma \ref{lem:hcirc} we have $h^\circ(x^-) = g(x^-)$ and $h^\circ(x^+) = g(x^+)$, so the optional sampling theorem (Proposition \ref{pro:optional_stopping}) 
gives
        \[
        h^\circ(x) = \rho(h^\circ(X_{\tau^*}^x))
        = \rho(g(X_{\tau^*}^x)).
        \]
Then since $g \leq h^\circ \in \mathcal{H}$, we have
\begin{equation*}
    w(x) 
    \leq h^\circ(x) 
    = \rho(g(X_{\tau^*}^x)) \leq \check V(x).
\end{equation*}
Conversely, if $\tau \in \mathscr{T}$ and $\tau \leq \tau^*$ then 
 $h^\circ(x) = \rho(h^\circ(X_\tau^x)) \geq \rho(g(X_\tau^x))$, giving $h^\circ(x) = \check V(x) = V(x)$ by Proposition \ref{cor:vvc}. To complete the proof note that for any $\tilde h = h^{y,z}_{\beta,\gamma} \in \mathcal{H}$ with $x \in [y,z]$ and $\tilde h \geq g$ we have 
\begin{align*}
\tilde h(x) &= \rho(\tilde h (X_{\tau^* \wedge \tau_{y,z}}^x))
\geq \rho(g(X_{\tau^*}^x))
= h^\circ(x),
\end{align*}
(cf. the last inequlity in \eqref{eq:nostep}) so that $w(x) \geq h^\circ(x)$ and hence
$w(x) 
= h^\circ(x) = V(x)$.
\label{end_of_proof}
\end{proof}

\subsection{Stopping set and regularity}
\label{subsec:stopset}

Since the proof of Theorem \ref{pro:representation_4param} concludes with $h^\circ(x)=V(x)$, we have 
\begin{corollary}
    The stopping policy 
\[
\tau^* = \tau_{x^-,x^+} = \inf\{t \geq 0: g(X_t^x)=V(X_t^x)\}
\]
is optimal in \eqref{optimal_stoping_intro}.
\end{corollary}
In this section we characterise the {\it stopping set}, the set of points $S=\{x \in [0,1]: V(x)=g(x)\}$ at which stopping is optimal, or equivalently the {\it continuation set} $C = [0,1] \setminus S$. We establish the continuity of the value function $V$ (Proposition \ref{pro:vcont}) and use the characterisation of $S$ (Lemma \ref{lem:charS}) to show in Corollary \ref{cor:sf} that differentiability of $g$ and $\mathcal{H}$ are sufficient for the differentiability of $V$. 

\begin{proposition}\label{pro:vcont}
    $V$ is continuous.
\end{proposition}

\begin{proof}
    The upper semicontinuity of $V$ was noted at the beginning of Section \ref{subsec:assum}. If $V(x) = g(x)$ then the lower semicontinuity of $V$ follows from the continuity of $g$ and the fact that $V \geq g$. If $V(x) > g(x)$ then since $g(x^\pm)=V(x^\pm)$ (Proposition \ref{lem:wgV}) we have $x^- < x < x^+$. Since $V = h^\circ$ on $\xpm$ it follows that $V$ is continuous on $(x^-,x^+)$.
\end{proof}

\begin{lemma}\label{lem:charS}
    Suppose that $g$ and $\mathcal{H}$ are differentiable. Then $x$ lies in the stopping set if and only if at least one of the following holds:
    \begin{enumerate}
        \item There exists $h \in \mathcal{H}$ with: \label{sf1}
        \begin{align}\label{eq:yconds}
        \begin{cases}
    h(0) < \infty \\
    h(x) = g(x) \\
    h'(x) = g'(x) \quad \text{ if } x>0\\
    h \geq g \text{ on } [0,1],            
        \end{cases}
        \end{align}
        \item There exists $h \in \mathcal{H}$ with: \label{sf2}
        \begin{align}\label{eq:zconds}
        \begin{cases}
    h(1) < \infty \\
    h(x) = g(x) \\
    h'(x) = g'(x) \quad \text{ if } x<1\\
    h\geq g \text{ on } [0,1].            
        \end{cases}
        \end{align} 
    \end{enumerate}
\end{lemma}

\begin{proof} 
    Condition \ref{sf1} is sufficient to imply $x \in S$, as is condition \ref{sf2}, so we assume that $x \in S$ and seek to establish the necessity of either condition \ref{sf1} or \ref{sf2}.

    We first construct $h_* \in \mathcal{H}$ with $h_* \geq g$ and $h_*(x) = g(x)$. Let $(\epsilon_n)_{n \in \N} \downarrow 0$ with $\epsilon_0 < \frac 1 2$. Since $x \in S$ implies that $w(x)=g(x)$, for each $n \in \N$ there exists $h_n \in \mathcal{H}$ with $h_n \geq g$ and $h_n(x) < g(x) + \epsilon_n$. By the construction of $\mathcal{H}$ (cf. Definition \ref{def:hcal}) and the monotonicity of $\varrho$, either $h_n(0) \leq g(x)+ \epsilon_n$ infinitely often or $h_n(1) \leq g(x)+\epsilon_n$ infinitely often. Without loss of generality, by the Bolzano-Weierstrass theorem (and choosing an appropriate subsequence in $\N$ if necessary) we may suppose that the sequence $(h_n(0))_n$ converges, so that $\beta := \lim_n h_n(0) \leq g(x)$. If $x=1$ then, by `continuity in $(\beta, \gamma)$', the functions $h^{0,1}_{h_n(0),h_n(1)}$ converge pointwise to $h_*:= h^{0,1}_{\beta,g(1)} \in \mathcal{H}$ with $h_* \geq g$ and $h_*(1)=g(1)$.
    
    Suppose therefore that $x<1$ and, with the convention $\inf \emptyset = \infty$, set 
    \[
    z_n := \inf\{z \in (x,1]: h_n(z) > \bar g + 1\} \wedge 1.
    \]
    By again choosing a suitable subsequence we may assume that $(z_n)_n$ converges to some $z' \leq 1$. Noting that $h_n(x) < \bar g + \frac 1 2$, we may appeal to the $\Delta-$equicontinuity of $\mathcal{H}$ as in the proof of Proposition \ref{lem:wgV} to conclude  that $z'>x$. Appealing again to $\Delta$-equicontinuity, there exists $z \in (x,z')$ such that $h_n(z) > \bar g + \frac 1 2$ for all $n$.
    Then by `continuity in $(\beta,\gamma)$', the functions $h^{0,z}_{h_n(0),h_n(z)}$ converge along some subsequence of $\N$ to $h_*:= h^{0,z}_{\beta,\gamma}$ for some $\gamma \in [\bar g + \frac 1 2, \bar g + 1]$ with $h_*(x) = g(x)$. If $z < 1$ then $h_*$ may be extended to an element of $\mathcal{H}$ (cf. the proof of Lemma \ref{lem:hcirc}). Assume that this extension has been performed.
    Then $h_*$ lies in $\mathcal{H}$ and monotonicity gives $h_* \geq g$ on $[0,1]$. If $x \in (0,1)$, this tangency implies that  $h'_*(x) = g'(x)$, so that condition \ref{sf1} of the lemma is satisfied. An analogous proof shows that if $\sup_n h_n(1) \leq g(x)$ then condition \ref{sf2} of the lemma is satisfied.
    \end{proof}
            
\begin{corollary}[Smooth fit principle]\label{cor:sf}
    If $g$ and $\mathcal{H}$ are differentiable then $V$ is differentiable on $(0,1)$. In particular $V'(x) = g'(x)$ for $x \in S \cap (0,1)$.
\end{corollary}

\begin{proof}
    For each $x \in S \cap (0,1)$, taking $h$ as in Lemma \ref{lem:charS} with $h(x)=g(x)$, $h'(x)=g'(x)$ and $h \geq g$ we have $h(x)=g(x)=V(x)$, so Theorem \ref{pro:representation_4param} gives $g(x) \leq V(x) = w(x) \leq h(x)$ and     
    \begin{align*}
    g'(x) & = \lim_{\epsilon \to 0}\frac{g(x+\epsilon) - g(x)}{\epsilon} \leq
    \lim_{\epsilon \to 0} \frac{V(x+\epsilon) - V(x)}{\epsilon} \\ & \leq
    \lim_{\epsilon \to 0} \frac{h(x+\epsilon) - h(x)}{\epsilon} = h'(x) = g'(x),
    \end{align*}
    so that $V'(x) = g'(x)$. 
    
    For each connected component $C_*$ of $C$ and each $x_* \in C_*$, let $x_*^-,x_*^+$ be defined as in \eqref{x_minus}--\eqref{x_plus}. Then $g(x_*^\pm)=V(x_*^\pm)$ (Lemma \ref{lem:parti}) and for any $x_*^- < x < x_*^+$ we have $x^\pm = x_*^\pm$. Then for $h^\circ$  defined as in Lemma \ref{lem:hcirc}, we have $V(x) = h^\circ(x) = h^{x^-,x^+}_{g(x^-),g(x^+)}(x) = h^{x_*^-,x_*^+}_{g(x_*^-),g(x_*^+)}(x) =: h_*^\circ (x)$. Thus $V = h_*^\circ$ on $C_*$, hence the differentiability of $V$ on $C_*$ follows immediately from that of $\mathcal{H}$ (Assumption \ref{ass:alg}).
\end{proof}

\subsection{Solution algorithm}
\label{subsec:algorithm}
The following algorithm aims to construct the value function globally by returning the set of continuation regions - that is, the connected components of the continuation region $C$. It exploits the fact, due to the smooth fit principle (Corollary \ref{cor:sf}) and the proof of Theorem \ref{pro:representation_4param}, that for each $x \in C$ the value function is tangent to the gain function at both $x^-$ and $x^+$. Then since $V = h^{x^-,x^+}_{g(x^-),g(x^+)}$ on $\xpm$, this double tangency can be used to search for $x^-$ and $x^+$. The computational cost is thus on the order of a search over $[0,1]^2$. In order to terminate in finite time with the correct result, we make the following additional assumptions:

\begin{assumption}\label{ass:alg}
\begin{enumerate}
    \item $\delta > 0$ is strictly smaller than the minimum length of the connected components of $C$, \label{aa2}
    \item the gain function $g$ and $\mathcal{H}$ are differentiable, \label{aa3}
    \item (strong monotonicity) if $X,Y \in b\mathcal{F}$ with $X \leq Y$ a.s.\ and $\PP(X < Y) > 0$ then $\rho(X) < \rho(Y)$. \label{aa4}
\end{enumerate}
\end{assumption}

Recall from the proof of Theorem \ref{pro:representation_4param} that if $x \in C$ then 
there exist $0 \leq x^- < x < x^+ \leq 1$ such that $V = h^{x^-,x^+}_{g(x^-),g(x^+)}$ on $\xpm$. Denoting by $h'^+$, $h'^-$ the right and left derivatives of $h$ respectively, we obtain the following lemma.
\begin{lemma}\label{lem:fulls}
Setting
\begin{align*}
F(x) &:= \big\{(y,z): 0 \leq y < x < z \leq 1, \; h = h^{y,z}_{g(y),g(z)} \text{ satisfies } \\  
& \qquad y(h'^+(y) - g'(y))=(1-z)(h'^-(z) - g'(z))=0\big\}, \\        
    G(x) &= \sup\{h^{y,z}_{g(y),g(z)}(x): (y,z) \in F(x)\}, \\
    \tilde{F}(x) &= \arg\max\{h^{y,z}_{g(y),g(z)}(x): (y,z) \in F(x)\},
\end{align*}
(with the convention that $\sup \emptyset = -\infty$) we have:
\begin{enumerate}
    \item $x \in S$ iff  $G(x) \leq g(x)$, \label{al1}
    \item $x \in C$ iff  $G(x) > g(x)$ and then 
    \begin{gather*}    
    x^- = \sup\{y: (y,z) \in \tilde{F}(x)\}, \quad  
    x^+ = \inf\{z: (y,z) \in \tilde{F}(x)\}, \\
    V = h^{x^-,x^+}_{g(x^-),g(x^+)} \text{ on } \xpm.
\end{gather*} \label{al2}
\end{enumerate}
\end{lemma}

\begin{proof}
For part \ref{al1} note that if $x \in C$ then $G(x)=V(x) > g(x)$, giving the implication $G(x) \leq g(x) \implies x \in S$. Conversely, if $x \in S$ then by optimality we have $G(x) \leq g(x) = V(x)$.

For part \ref{al2}, by the proof of Theorem \ref{pro:representation_4param} and the smooth fit principle we have $(x^-, x^+) \in \tilde{F}(x)$, and it is straightforward to show that $x^- = \sup\{y: (y,z) \in \tilde{F}(x)\}\}$ and $x^+ = \inf\{z: (y,z) \in \tilde{F}(x)\}\}$, as follows (we show only the latter). Suppose that $(y,z) \in \tilde{F}(x)$. Then if $z < x^+$ we have either $y > x^-$ or $y \leq x^-$. 

If $x^- < y < x < z < x^+$, by construction we have $\min\{V(y)-g(y),V(z)-g(z)\}>0$ and monotonicity and the $\varrho$-martingale property give 
\[
h^{y,z}_{g(y),g(z)}(x) < h^{y,z}_{V(y),V(z)}(x) =
V(x) = h^{x^-,x^+}_{g(x^-),g(x^+)}(x),
\]
contradicting the definition of $\tilde{F}(x)$. 

Supposing alternatively that $y \leq x^- < x < z < x^+$, we again have $V(z)>g(z) = h^{y,z}_{g(y),g(z)}(z)$. 
Now if $h^{y,z}_{g(y),g(z)}(x^-) \leq V(x^-)$ then monotonicity and part \ref{aa3} of Assumption \ref{ass:alg} give $h^{y,z}_{g(y),g(z)}(x) < V(x)$, again contradicting the definition of $\tilde{F}(x)$. Thus $h^{y,z}_{g(y),g(z)}(x^-) > V(x^-) = g(x^-)$, which contradicts the definition of $V$. We conclude that $z \geq  x^+$.
\end{proof}
Lemma \ref{lem:fulls} is the basis for Algorithm \ref{alg:2}.

\begin{algorithm}[!h]
\caption{Under Assumptions \ref{continuity_assumption} and \ref{ass:alg} this algorithm solves the optimal stopping problem by returning the set of connected components of the continuation region $C$.}
\label{alg:2}
\begin{algorithmic}
\State Evaluate 
\begin{align*}
F &:= \big\{(y,z): 0 \leq y < z \leq 1, \; h = h^{y,z}_{g(y),g(z)} \text{ satisfies } \\ 
& \qquad y(h'^+(y) - g'(y))=(1-z)(h'^-(z) - g'(z))=0\big\},
\end{align*}
\State Set $x = 0$, $B = \emptyset$
\While{$x \leq 1$}
\State Set
\begin{align*}
F(x) &:= \{(y,z) \in F: y < x < z \}, \\  
    G(x) &= \sup\{h^{y,z}_{g(y),g(z)}(x): (y,z) \in F(x)\}, \\
    \tilde{F}(x) &= \arg\max\{h^{y,z}_{g(y),g(z)}(x): (y,z) \in F(x)\},\\
    x^- &= \sup\{y: (y,z) \in \tilde{F}(x)\}, \\ 
    x^+ &= \inf\{z: (y,z) \in \tilde{F}(x)\}, \label{eq:msf}
\end{align*}
\If{$G(x) > g(x)$} 
    \State Append the element $(x^-, x^+)$ to $B$
    \State Set $x = x^+$
\EndIf
\State Increase $x$ by $\delta$
\EndWhile
\State \Return $B$
\end{algorithmic}
\end{algorithm}

\begin{remark}\label{rem:ctsfit}
    Part \ref{aa2} of Assumption \ref{ass:alg} ensures that the continuation region $C$ has finitely many (at most $\lfloor 1/\delta \rfloor$) connected components and, thus, that Algorithm \ref{alg:2} terminates in finite time. Since the assumptions made in this section are strong, it is worth remarking on the output of Algorithm \ref{alg:2} when Assumptions \ref{continuity_assumption} and \ref{ass:alg} do not hold. The algorithm identifies the points $x^+,x^-$ defined in \eqref{x_minus}--\eqref{x_plus} using the principle of smooth fit. It may be that for certain initial points $x_0$, the principle of smooth fit applies in a sufficiently large neighbourhood of $x_0$. An example is provided in Section \ref{sec:wcrevis} below using the worst-case risk mapping. In this case it is not difficult to see that Algorithm \ref{alg:2} correctly returns the value $V(x_0)$ and the optimal stopping time $\tau^*$ for initial points $x_0 \in (\frac 1 8, \frac 5 8)$. However for initial points $x_0 \in (\frac 3 4, 1)$, the lack of smooth fit at $x = \frac 3 4$ prevents the component $(\frac 1 8, \frac 5 8)$ of $C$ from being identified.

    Nevertheless the characterisation \eqref{x_minus}--\eqref{x_plus} of the points $x^+,x^-$ uses a  weaker property than smooth fit, namely {\it continuous fit} between $g$ and $V$. Since formula \eqref{eq:guessv} correctly constructs $V$ under the worst-case risk mapping, it seems possible to modify Algorithm \ref{alg:2} to use continuous fit instead to identify the points $x^+,x^-$. While outside the scope of the present paper, such a modification would enable the algorithm to solve the example of Section \ref{sec:wcrevis} for all $x \in C$.
\end{remark}

\section{Examples}
\label{sec:gt}

In this section we consider optimal stopping under three different risk mappings $\rho$. For a unified framework, consider the construction
\begin{equation}
\label{dual}
\rho(Z) := \inf_{Q\ll \PP} \left( \E^Q[Z]+\alpha(Q) \right), \qquad Z \in b\mathcal{F},
\end{equation}
where for each $Q \ll \PP$, $\alpha(Q) \in [0,\infty]$ is a `penalty' associated with $Q$ such that $\inf_{Q \ll \PP}\alpha(Q)=0$, so that $\rho$ is indeed a risk mapping (which is also concave, see e.g. \cite[p. 207]{Follmer_Schied}). Thus \eqref{dual} may be used directly to construct a risk mapping by specifying the penalty function $\alpha$, and this is the approach which will be taken below. It is interesting to note that further examples of risk mappings of the form \eqref{dual} can be constructed from the solutions to backwards stochastic differential equations (BSDEs, see for example \cite{barrieu2009pricing,chong2019ergodic}), and such examples are time consistent by construction.

Section \ref{subsec:linear} recovers the classical linear solution of \cite[Ch.\ III.7]{Dynkin_Yushkevich} and Section \ref{sec:wcrevis} revisits the worst-case risk mapping of Section \ref{sec:motiv}, while Section \ref{subsec:entropic} addresses the entropic risk mapping. Unless otherwise stated, we take the underlying diffusion $Y$ to be standard Brownian motion.

\subsection{Linear case}
\label{subsec:linear}
Setting $\alpha(Q)=0$ for $Q=\PP$ and $\alpha(Q)=+\infty$ otherwise gives the linear expectation $\rho(Z) = {\E}[Z]$, which has the corresponding dynamic risk mapping
\begin{align*}
    \rho_t(Z) = {\E}[Z\vert \mathcal{F}_t], \qquad t\geq 0.
\end{align*}
In this case, $\mathcal{H}$ consists of line segments:
\begin{equation}\label{h_beta_gamma_linear}
    h^{y,z}_{\beta,\gamma}(x) = \beta \PP(\tau_y^x<\tau_z^x) + \gamma \PP(\tau_z^x<\tau_y^x) = \beta \frac{z-x}{z-y} + \gamma \frac{x-y}{z-y}, \qquad x \in [y,z],
\end{equation}
(see for example \cite[Th.\ 30]{Serfozo}). Thus $\mathcal{H}$ is differentiable, and $\Delta$-equicontinuous by Lemma \ref{lem:diffislec}. Also, $\varrho$ is continuous by the conditional bounded convergence theorem for $\E$ and time consistent by its conditional tower property. 
The strong Markov property (Definition \ref{def:mp}) is trivially satisfied as Brownian motion is a strong Markov process under $\E$.

Thus Assumption \ref{continuity_assumption} holds, and Theorem \ref{pro:representation_4param} gives that if $g$ is continuous then $V=w$ (cf. \eqref{eq:repV}). It is straightforward to check that $w$ is then equal to the pointwise infimum over those $h \in \mathcal{H}^\dagger$ (cf. \eqref{eq:hdag}) which dominate $g$ or, equivalently, that $w$ is the smallest concave majorant of $g$, and we recover the classical geometric method for optimal stopping problems \cite[Ch.\ III.7]{Dynkin_Yushkevich}.

More generally if $X^x$ is a one-dimensional diffusion, stopped as in Section \ref{sec:setting}, then we recover the result of Dayanik and Karatzas for undiscounted optimal stopping on a bounded interval (see \cite{Dayanik_Karatzas}). To see this write $Z_t^{S(x)}=S(X_t^x)$, where $S$ is the scale function of $X^x$, and note from part \ref{def:hcal1} of Definition \ref{def:hcal} that the set $\mathcal{H}$ of generalised potentials depends implicitly on the law of the process $X^x$. Thus, for clarity we may denote by $\mathcal{H}_X$ and $\mathcal{H}_Z$ the sets of generalised potentials induced respectively by the processes $X^x$ and $Z^{S(x)}=S(X^x)$. According to \cite[Prop.\ 2.2]{Dayanik_Karatzas}, the elements of $\mathcal{H}_X$ take the form
\begin{equation}\label{eq:hx}
    h^{y,z}_{\beta,\gamma}(x) = \beta \frac{S(z)-S(x)}{S(z)-S(y)} + \gamma \frac{S(x)-S(y)}{S(z)-S(y)}, \qquad x \in [y,z].
\end{equation}
Then since $X^x=S^{-1}\left(Z^{S(x)}\right)$, it follows from the definition of $h^{y,z}_{\beta,\gamma}$ via hitting probabilities (Definition \ref{def:hcal})
that the elements of $\mathcal{H}_Z$ take the form
\begin{equation*}
    h^{y,z}_{\beta,\gamma}(x) = \beta \frac{z-x}{z-y} + \gamma \frac{x-y}{z-y}, \qquad x \in [y,z],
\end{equation*}
which is the same as $\eqref{h_beta_gamma_linear}$, and the conditions of Theorem \ref{pro:representation_4param} are again satisfied. Rewriting problem \eqref{optimal_stoping_intro} as 
$V(x)=\sup_{\tau\in\mathscr{T}} \E\left[ g\circ S^{-1}\left(Z^{S(x)}_\tau \right)\right]$, we conclude that its solution is the smallest  concave majorant of $g\circ S^{-1}$, which is \cite[Prop.\ 3.3]{Dayanik_Karatzas}.

\subsection{Worst-case risk mapping}
\label{sec:wcrevis}
As discussed for example in \cite{denis2011function}, where further references can be found, many `worst case' risk mappings are naturally associated with settings involving ambiguity: for example, when the drift function $\mu$ or the volatility function $\sigma$ is ambiguous. These respective settings can be treated in two frameworks due to Peng: the former using the $g-$expectation \cite{peng1997backward} (see also \cite{chen2002ambiguity}),
and the latter with $G-$Brownian motion and $G-$expectation \cite{peng2007g, peng2007gb}. In this section we consider a definition of `worst case' which is not related to ambiguity: namely, when the drift and volatility functions are known and the essential infimum of the payoff is considered (this was the nonlinear expectation considered in the example of Section \ref{sec:motiv}). This can be constructed by setting 
$\alpha(Q)=0$ in \eqref{dual} for all $Q\ll\PP$, giving $\rho(Z) = \PP - \essinf (Z)$
\cite[Remark 3.7]{barron2000radon}. Although the worst-case risk mapping does not satisfy Assumption \ref{continuity_assumption}, it was observed in Section \ref{sec:motiv} that formula \eqref{eq:guessv} generates the correct solution and, as such, is a useful heuristic. This heuristic is novel to the best of our knowledge: for example, it differs from the use of the simpler set $\mathcal{H}^\dagger$ (cf. \eqref{eq:hdag}). Indeed it is easy to see that if $\mathcal{H}$ is replaced by $\mathcal{H}^\dagger$ in \eqref{eq:repV}, instead of $w$ we obtain the constant function $w^\dagger(x) \equiv \bar g$, and for all $x \notin \arg\sup_{y \in [0,1]}g(y)$ we have
\[
V(x) = w(x) = \min\{\sup_{y \in [0,x]}g(y), \sup_{y \in [x,1]}g(y)\} \neq w^\dagger(x) = \bar g.
\]
Formula \eqref{eq:guessv} also differs from the heuristic use of the principle of smooth fit (cf. Corollary \ref{cor:sf}). This can be seen in Figure \ref{fig:4}, where the gain function is $g(x) = 1 + \sin (4 \pi x)$ and, as observed in Remark \ref{rem:ctsfit}, the use of smooth fit fails for initial points $x_0 \in (\frac 3 4, 1)$.

\subsection{Entropic case}
\label{subsec:entropic}
Setting $Y := \frac{\ud Q}{\ud \PP}$ and $\alpha(Q)=\E^Q\left[ \ln Y \right]$ in \eqref{dual}, one obtains the entropic risk mapping \cite[Ex. 4.34]{Follmer_Schied}:
\begin{align*}
    \rho(Z) &= - \ln \E[e^{-Z}], \qquad Z \in b\mathcal{F},
\end{align*}
which has the corresponding dynamic conditional risk mapping 
\begin{align*}
    \rho_t(Z) &= - \ln \E[e^{-Z} \vert \mathcal{F}_t], \qquad t \geq 0.
\end{align*}
(For an alternative construction of the entropic risk mapping via quadratic BSDEs see, for example, \cite[Proposition 6.4]{barrieu2009pricing}, \cite{chong2019ergodic, delbaen2011uniqueness}).
From \eqref{h_beta_gamma_linear}, for $x \in [y,z]$ we have $h_{\beta,\gamma}^{y,z}(x) 
= -\ln \left( e^{-\beta}\frac{z-x}{z-y} + e^{-\gamma}\frac{x-y}{z-y} \right)$, so $\mathcal{H}$ is again differentiable, and $\Delta$-equicontinuous by Lemma \ref{lem:diffislec}. The continuity, strong Markov property and time consistency of $\varrho$ follow from the corresponding properties for $\E$ as in Section \ref{subsec:linear}, so Assumption \ref{continuity_assumption} holds and Theorem \ref{pro:representation_4param} gives that if $g$ is continuous then $V=w$. 

In this case it is a priori equivalent to rescale the problem and instead maximise the linear expectation $\E[e^{-g(X_\tau^x)}]$, for which the classical construction using $\mathcal{H}^\dagger$ suffices. Interestingly, though, the direct construction of the present paper is qualitatively different, in the sense that the set $\mathcal{H}$ must again be used rather than $\mathcal{H}^\dagger$. Indeed, fix $x \in (0,1)$ and let $g$ be any gain function with $g(0)=1$ and $g(x) > 2 -\ln (1-x)$. We have $h_{0,\gamma}^{0,1}(x) \uparrow -\ln (1-x)$ as $\gamma \uparrow \infty$.  Then if $\mathcal{H}^\dagger$ is used to construct $w^\dagger$ (rather than using $\mathcal{H}$ to construct $w$), by part \ref{ph4} of Lemma \ref{lem:hprops} (translation invariance) we would have $w^\dagger(0) \geq 2 > 1 = w(0) = V(0) = g(0)$, so $w^\dagger$ cannot be the value function.

\subsection{\bl{A sufficient condition for Algorithm \ref{alg:2}}}
\bl{
We conclude by providing an example problem for which Assumption \ref{ass:alg} holds and, thus, Algorithm \ref{alg:2} may be applied. Our focus is on part \ref{aa2} of Assumption \ref{ass:alg}, since the remaining parts are straightforward. 

 Suppose that each $h \in \mathcal{H}$ is a 
 convex function on $d(h)$. Examples include the linear case of Section \ref{subsec:linear}, the entropic case of Section \ref{subsec:entropic}, and further examples which can be constructed using the dual representation \eqref{dual}, as follows. From \eqref{eq:hx}, the potentials $h \in \mathcal{H}_X$ are convex precisely when the scale function $S$ of $X^x$ is convex. (This holds, for example, when $X^x$ has negative drift.) Let $\mathcal Q$ be a finite set of measures $Q \ll \PP$, each corresponding to a regular diffusion with a convex scale function. If the penalty function $\alpha$ is finite on $\mathcal Q$, and is set to $+\infty$ otherwise, the resulting concave risk mapping $\rho$ also has convex potentials $\mathcal{H}$.
 
 Let $C_0=(x_0,y_0)$ be a connected component of the continuation region $C$. Then since $V$ dominates $g$ and is convex on $C_0$ with $V=g$ at its endpoints, $g$ cannot be concave on the whole of $C_0$ (otherwise $V=g$ on $C_0$, a contradiction). Thus $g$ is convex on some subinterval of $C_0$. Hence a sufficient condition for part \ref{aa2} of Assumption \ref{ass:alg} is that $\mathscr{A}$ is finite, where $\mathscr{A}$ is the set of (maximal) intervals on which $g$ is convex, and then we may take $\delta < \min_{[x_1,y_1] \in \mathscr{A}}(y_1-x_1)$. Note that $\delta$ can be found from knowledge of $g$ only.
 }

\begin{acks}[Acknowledgments]
The authors would like to thank Giorgio Ferrari, \bl{Ben Hambly,} Vicky Henderson, David Hobson and Jan Palczewski for comments on earlier versions of this work, and also two anonymous referees whose feedback significantly improved the manuscript.
\end{acks}

\begin{funding}
This work was supported by the Lloyd’s Register Foundation-Alan Turing Institute programme on Data-Centric Engineering under the LRF grant G0095.
\end{funding}

\bibliographystyle{imsart-number} f Style BST file (imsart-number.bst or imsart-nameyear.bst)
\bibliography{bibliography}       


\end{document}